\documentclass[a4paper,11pt]{article}

\usepackage[latin1]{inputenc}
\usepackage[T1]{fontenc}

\usepackage{array}
\usepackage{fancyhdr}
\usepackage{longtable}
\usepackage{rotating}

\usepackage{makeidx}
\usepackage{amsmath}
\usepackage{amsthm}
\usepackage{amssymb}
\usepackage[mathscr]{eucal}
\usepackage{enumerate}
\usepackage[dvips]{epsfig}
\usepackage{float}  
\usepackage{ifthen} 
\usepackage{extarrows}  
\usepackage{stmaryrd}   
\usepackage{arydshln}     
\usepackage{multirow}
\usepackage{mathtools}
\usepackage[dvips]{geometry} 
\usepackage{bbm} 

\usepackage{tabularx}
\newcolumntype{L}[1]{>{\raggedright\arraybackslash}p{#1}} 
\newcolumntype{C}[1]{>{\centering\arraybackslash}p{#1}} 
\newcolumntype{R}[1]{>{\raggedleft\arraybackslash}p{#1}} 

\usepackage{listings}
\usepackage{subfig}
\usepackage{dsfont}
\usepackage{xspace}

\theoremstyle{plain}
\newtheorem{thm}{Theorem}[section]
\newtheorem{prop}[thm]{Proposition}

\newtheorem{lem}[thm]{Lemma}
\newtheorem{cor}[thm]{Corollary}
\theoremstyle{definition}

\theoremstyle{remark}

\newcommand{\Var}{\operatorname{Var}}

\newcommand{\Prob}{\mathrm{P}}
\newcommand{\Erw}{\mathrm{E}}

\newcommand{\ind}{\mathbbmss{1}}

\newcommand{\Nd}{\operatorname{N}}	

\newcommand{\diag}{\operatorname{diag}}

\newcommand{\costs}{\operatorname{cost}}

\newcommand{\Oo}{\mathcal{O}}
\DeclareMathOperator{\vecop}{vec}
\DeclareMathOperator{\matop}{mat}

\newcommand{\Mid}{\; | \;}
\newcommand{\bigMid}{\; \big| \;}
\newcommand{\biggMid}{\; \bigg| \;}

\newcommand{\Tt}{\mathsf{T}}

\newcommand{\IA}{\text{IA}}
\newcommand{\WIK}{\text{WIK}}
\newcommand{\KPW}{\text{FS}}

\newcommand{\EM}{\text{EM}}

\newcommand{\Se}{n}		

\usepackage[bookmarksopen,bookmarksnumbered,colorlinks=true,linkcolor=red,citecolor=blue]{hyperref}
\usepackage{bookmark}
\hypersetup{
	pdfauthor={Andreas R\"o{\ss}ler},
	% ...
}

\title{{On the Approximation and Simulation of Iterated Stochastic Integrals 
		and the Corresponding L\'{e}vy Areas 
		in Terms of a Multidimensional Brownian Motion}}

\author{Jan Mrongowius %
	and Andreas R\"o{\ss}ler\thanks{e-mail: roessler@math.uni-luebeck.de, \quad
	January 21, 2021}
\bigskip
\\
\small{
	Institute of Mathematics, Universit\"at zu L\"ubeck,} \\
\small{Ratzeburger Allee 160, 23562 L\"ubeck, Germany} 
}

\date{}

\begin{document}

\allowdisplaybreaks

\maketitle

\begin{abstract}
\noindent
A new algorithm for the approximation and simulation of twofold iterated
stochastic integrals together with the corresponding L\'{e}vy areas driven 
by a multidimensional Brownian motion is proposed. 
The algorithm is based on a 
truncated Fourier series approach. However, the approximation of the remainder
terms differs from the approach considered by Wiktrosson~(2001). As the main
advantage, the presented algorithm makes use of a diagonal covariance matrix
for the approximation of one part of the remainder term and has a higher 
accuracy due to an exact approximation of the other part of the remainder. This
results in a significant reduction of the computational cost compared to, e.g., 
the algorithm introduced by Wiktorsson. Convergence in $L^p(\Omega)$-norm 
with $p \geq 2$ for the approximations calculated with the new algorithm as well 
as for approximations calculated by the basic truncated Fourier series algorithm is 
proved and the efficiency of the new algorithm is analyzed.
\end{abstract}
%
%
%
%
\section{Introduction}
Iterated stochastic integrals play an important role in the context of stochastic
processes driven by Brownian motion. They appear, e.g., in stochastic Taylor
expansions~\cite{MR1214374}, the Wiener chaos expansion~\cite{MR2200233}
and they are needed for approximation schemes of higher orders like the Milstein
scheme~\cite{Mil75,Mil95} or stochastic Runge-Kutta methods~\cite{MR2669396}
for stochastic differential equations (SDEs) as well as for stochastic partial
differential equations (SPDEs)~\cite{HaRoe20arXiv2006.08275v1,MR3320928}.
The approximation and simulation of iterated stochastic integrals driven by
a multidimensional Brownian motion is still subject
of ongoing research and there is a demand for efficient algorithms. 
Milstein (1995)~\cite{Mil95} as well as Kloeden, Platen and Wright
(1992)~\cite{MR1178485}, see also \cite{MR1214374},
proposed an approximation algorithm for twofold iterated stochastic integrals 
in case of an $m$-dimensional driving Brownian motion for arbitrary $m \geq 1$
that is based on a truncated Fourier series expansion
of the Brownian bridge process and they prove the exact $L^2$-error for that
approach. However, the convergence rate of the truncated Fourier series approach 
is not good enough for a reasonable application with, e.g., the Milstein scheme  
compared to the Euler-Maruyama scheme~\cite{MR1214374,Mil95} if the
computational cost are taken into account in the finite dimensional SDE setting.
It has to be mentioned that Gaines and Lyons (1994)~\cite{MR1284705} proposed
an algorithm for the special case of a 2-dimensional driving Brownian motion 
based on a generalization of Marsaglia's rectangle-wedge-tail method that is, however, complicated to implement. Moreover, in the papers \cite{MR3178331} and \cite{MR2116088} the authors approximate the 
inverse of the distribution function by polynomials in order to simulate realizations
of the twofold iterated stochastic integral in the case of a 2-dimensional driving
Brownian motion.
\\ \\
For the general $m$-dimensional case with $m \geq 1$, Wiktorsson (2001)~\cite{MR1843055}
improved the Fourier series 
approach by taking into account an approximation of the truncation term. 
Wiktorsson's idea is a break through on the way to increase the efficiency 
of such approximation algorithms and allows for an improvement of the strong 
order of convergence of, e.g., the Milstein scheme compared to the Euler 
Maruyama scheme in case of SDEs driven by a multidimensional Brownian motion.
Recently, Pleis (2020)~\cite{PleisDiss20} proposed an improved algorithm that 
is based on the approach of Wiktorsson~\cite{MR1843055}. This algorithm is of the
same order of convergence as that of Wiktorsson, however Pleis obtained some
smaller constant for his error estimate and thus needs less 
computational effort compared to Wiktorsson's algorithm. Moreover, convergence
in $L^p$-norm for $p \geq 2$ is proved for this algorithm.
For the infinite dimensional setting, the algorithm proposed in
\cite{MR1178485,MR1214374,Mil95} as well as 
the algorithm in \cite{MR1843055} have been generalized by Leonhard and
R\"o{\ss}ler (2019)~\cite{MR3949104} for the approximation of iterated 
stochastic integrals driven by a $Q$-Wiener process like
they are needed, e.g., in the setting of SPDEs. We refer to
\cite{HaRoe20arXiv2006.08275v1,MCQMC2018} and \cite{MR3320928} for 
a detailed discussion of this topic for Milstein type schemes applied to SPDEs
without commutative noise.
\\ \\
In the present paper, we propose an algorithm for the approximation of twofold 
iterated stochastic integrals driven by a multidimensional Brownian motion.
This new algorithm is based on the Fourier series approach and makes use of 
the seminal idea due to Wiktorsson~\cite{MR1843055} to approximate the 
truncation term by some appropriate multivariate Gaussian random variable.
However, in 
contrast to Wiktorsson's algorithm, we split the truncation term into two parts.
The first part can be exactly simulated using Gaussian random
variables, see \cite{MR1178485,MR1214374,Mil95}, while that of the second part 
is conditionally Gaussian and approximated by replacing the exact conditional
covariance matrix by a deterministic diagonal matrix similar to the approach 
by Wiktorsson~\cite{MR1843055}. The idea to split the truncation term was first
considered by Milstein~\cite{Mil95} as well as by Kloeden, Platen and
Wright~\cite{MR1178485}. Pleis~\cite{PleisDiss20} combined this idea together 
with Wiktorsson's idea for the approximation of the truncation terms. 
In contrast to the approach by Pleis, we combine these two ideas in a different 
way in order to minimize conditional dependencies of the appearing random
variables.
As a result of this, we derive a different algorithm with better error estimates and 
thus less computational effort. Compared to Wiktorsson's algorithm, the main
advantage of the newly proposed algorithm are error estimates with some smaller
constant and that the covariance matrix for the approximate simulation of the
truncation error is a constant diagonal matrix. This has significant consequences 
if the iterated stochastic integrals need to be simulated on many time intervals of
some time discretization for, e.g., a numerical scheme like the Milstein scheme. 
If Wiktorsson's algorithm is applied, one has to recalculate the square root of the
covariance matrix for the approximation of the truncation term each time step
as it depends on the increments of Brownian motion. In contrast to this,
the covariance matrix of our algorithm is simply the identity matrix multiplied by 
some constant. Thus, it is not necessary to recalculate the covariance matrix 
and to calculate the square root of this matrix for each time step. Furthermore, 
due to the smaller error constant less normally distributed random variables are
needed by the proposed algorithm. Beside
enormous savings of computational cost the newly proposed approach allows 
for a dynamical scaling of the dimension of the driving Brownian motion without
a recalculation of already derived approximations due to the diagonal structure 
of the covariance matrix.
\\ \\
The paper is organized as follows: In Section~2, we briefly describe the 
Fourier series expansion of the Brownian bridge process, which leads to the
approximation algorithm proposed by Milstein~\cite{Mil95} and by Kloeden,
Platen and Wright~\cite{MR1178485}. Given their approach, we analyze
the conditional distribution of the truncation terms in Section~3, which turn
out to have a Gaussian distribution. This feature is used in Section~4 to
define the newly proposed approximation algorithm for twofold iterated 
stochastic integrals. Since the Hilbert space setting allows for sharper error
estimates, we first prove error estimates in $L^2$-norm in Section~4.1 whereas
$L^p$-norm estimates for $p>2$ are proved in Section~4.2 for the well known
truncated Fourier series approach by Milstein~\cite{Mil95} and by Kloeden, 
Platen and Wright~\cite{MR1178485} as well as for the newly proposed algorithm.
In contrast to the approximation problem, we consider the simulation of 
iterated stochastic integrals with the newly proposed algorithm in Section~5.1.
Moreover, we compare the efficiency and the computational cost of our new
algorithm with the the cost of the algorithm proposed by
Wiktorsson~\cite{MR1843055} and the algorithm due to Milstein~\cite{Mil95} 
and Kloeden, Platen and Wright~\cite{MR1178485}.
\section{A series expansion for L\'{e}vy areas}
In the following, let a complete probability space $(\Omega, \mathcal{F}, \Prob)$ 
be given and
let $(W_t)_{t \geq 0}$ denote an $m$-dimensional Brownian motion where
$W_t = (W_t^1, \ldots, W_t^m)^{\Tt}$ and $(W_t^i)_{t \geq 0}$ for $i=1, \ldots, m$
are independent scalar Brownian motions.
We want to approximate the iterated stochastic It\^o integrals
\begin{align} \label{Iter-Def-th}
	I_{(i,j)}(t,t+h) = \int_{t}^{t+h} \int_{t}^s \, \mathrm{d}W_u^i 
	\, \mathrm{d}W_s^j
\end{align} 
or, in case of Stratonovich calculus, the iterated stochastic Stratonovich integrals
\begin{align} \label{Iter-Strato-Def-th}
	J_{(i,j)}(t,t+h) = \int_{t}^{t+h} \int_{t}^s \, {\circ \mathrm{d}W_u^i}
	\, {\circ \mathrm{d}W_s^j}
\end{align} 
for $i, j \in \{1, \ldots, m\}$ and some $0 \leq t < t+ h < \infty$. The idea is to
consider a Fourier series 
expansion of the corresponding Brownian bridge process that can be used
to replace the integrand of the twofold iterated stochastic integral. Then, a 
detailed analysis for a suitable truncation of this Fourier series as well as an
approximation of the truncation term will be applied in order to develop 
the approximation algorithm that we propose in the following. 
\\ \\
Let $\Delta W^i(t,t+h) = W_{t+h}^i - W_t^i$ denote the increment of the $i$th
component of $(W_t)_{t \geq 0}$ for $i=1, \ldots, m$.
We point out that the 
joint distribution of $I_{(i,j)}(t,t+h)$, $\Delta W^i(t,t+h)$ and $W^j(t,t+h)$ for 
$i,j=1, \ldots, m$ does not depend on $t$. Therefore, w.l.o.g.\ we choose $t=0$
in the following and we simply write $I_{(i,j)}(h)$
and $\Delta W^i(h)$ for $I_{(i,j)}(t,t+h)$ and $\Delta W^i(t,t+h)$, respectively.
The same applies to the iterated Stratonovich integral $J_{(i,j)}(h)$. Further, let
$\Delta W(h) = (\Delta W^i(h))_{1 \leq i \leq m}$ denote the $m$-dimensional 
vector of the increments of the Brownian motion and let
$I(h) = (I_{(i,j)}(h) )_{1 \leq i,j \leq m}$ and $J(h) = (J_{(i,j)}(h))_{1 \leq i,j \leq m}$
denote the $m \times m$ matrices of the iterated stochastic It{\^o} and
Stratonovich integrals, respectively.
\\ \\
Twofold iterated stochastic integrals can be expressed by so-called 
L\'{e}vy stochastic area integrals and increments of the driving Brownian motions.
Therefore, the approximation of these iterated stochastic 
integrals is synonymous with the approximation of the  corresponding L\'{e}vy 
area. Let $A(h) = (A_{(i,j)}(h))_{1 \leq i,j \leq m}$ denote the $m \times m$ matrix
of the L\'{e}vy areas $A_{(i,j)}(h)$ that are defined as
\begin{align} \label{Def-Levy-Area}
	A_{(i,j)}(h) = \frac{1}{2} \big( I_{(i,j)}(h) - I_{(j,i)}(h) \big)
\end{align}
for $i,j = 1, \ldots, m$ (see, e.g., \cite{MR1284705}). Then, it holds that
\begin{align} 
	\Delta W^i(h) \, \Delta W^j(h) &= I_{(i,j)}(h) + I_{(j,i)}(h) \, , \\
	\label{Prop-Levy-Area-ij-ji}
	A_{(i,j)}(h) &= -A_{(j,i)}(h) \, , \\
	\label{Prop-Levy-Area-ii}
	A_{(i,i)}(h) &= 0
\end{align}
for $i \neq j$ and
\begin{align} \label{Prop-Iij}
	I_{(i,j)}(h) &= \frac{1}{2} \big( \Delta W^i(h) \, \Delta W^j(h) - h \, \ind_{i=j} \big)
	+ A_{(i,j)}(h) \,, \\
	 \label{Prop-Jij}
	J_{(i,j)}(h) &= \frac{1}{2} \big( \Delta W^i(h) \, \Delta W^j(h) \big)
	+ A_{(i,j)}(h)
\end{align}
for any $i,j \in \{1, \ldots, m\}$. Since $J_{(i,j)}(h) = I_{(i,j)}(h)$ for $i \neq j$
and $J_{(i,i)}(h) = I_{(i,i)}(h) + \frac{1}{2} h$, we can restrict our considerations 
to iterated stochastic It\^o integrals $I_{(i,j)}(h)$ in the following.
\\ \\
In order to obtain a Fourier series expansion of the Brownian motion 
$(W^i_t)_{t \geq 0}$, we start with the so-called Brownian bridge process
$(\tilde{W}^i_t)_{t \in [0,h]}$ defined by
\begin{align} \label{Brownian-bridge}
	\tilde{W}^i_t = W^i_t - \frac{t}{h} W^i_h \, , \quad t \in [0,h] \, .
\end{align}
Milstein~\cite{Mil95} as well as Kloeden, Platen and Wright~\cite{MR1178485}
considered the Fourier series expansion of the Brownian bridge process
$(\tilde{W}^i_t)_{t \in [0,h]}$ which results in
\begin{align} \label{BBridge-Series}
	W^i_t = \frac{t}{h} W^i_h + \frac{a_{i,0}}{2} + \sum_{k=1}^{\infty} a_{i,k} \, 
	\cos \bigg( \frac{2 \pi k t}{h} \bigg) + b_{i,k} \, \sin \bigg( \frac{2 \pi k t}{h} \bigg)
\end{align}
for $t \in [0,h]$ with random coefficients
\begin{align*}
	a_{i,k} &= \frac{2}{h} \int_0^h \Big( W^i_s - \frac{s}{h} W^i_h \Big) 
	\cos \bigg( \frac{2 \pi k s}{h} \bigg) \, \mathrm{d}s \, ,
	\quad \quad
	b_{i,k} = \frac{2}{h} \int_0^h \Big( W^i_s - \frac{s}{h} W^i_h \Big) 
	\sin \bigg( \frac{2 \pi k s}{h} \bigg) \, \mathrm{d}s
\end{align*} 
for $k \in \mathbb{N}_0$ and for $i \in \{1, \ldots, m\}$. 
The series on the right hand side of equation \eqref{BBridge-Series} converges 
in the  $L^2(\Omega)$-norm~\cite{MR1178485,MR1214374}.
Since the sample paths of a Brownian motion are almost surely continuous,
the right hand side of \eqref{BBridge-Series} converges also $\Prob$-a.s.\ in 
the $L^2([0,h])$-norm to $(W_t)_{t \in [0,h]}$.
The coefficients $a_{i,k}$ and $b_{i,k}$ as well as the increments $\Delta W^i(h)$
for $i \in \{1, \ldots, m\}$ and $k \in \mathbb{N}$ are all independent Gaussian 
random variables with $a_{i,k} \sim \Nd \big( 0, \frac{h}{2 \pi^2 k^2} \big)$,
$b_{i,k} \sim \Nd \big( 0, \frac{h}{2 \pi^2 k^2} \big)$ and $\Delta W^i(h) \sim 
\Nd(0,h)$.
Further, it holds $a_{i,0} = -2 \sum_{k=1}^{\infty}
a_{i,k}$ with $a_{i,0} \sim \Nd \big( 0, \frac{h}{3} \big)$.
\\ \\
Next, the integrand $W_s^i$ in \eqref{Iter-Def-th} is replaced by its 
Fourier series \eqref{BBridge-Series}. Integrating this expression results in
\eqref{Prop-Iij} with representation
\begin{align} \label{LevyArea-Series-Rep}
	A_{(i,j)}(h) &= \frac{1}{2} \big( a_{i,0} \,
	\Delta W^j(h) - a_{j,0} \, \Delta W^i(h) \big)
	+ \pi \sum_{k=1}^{\infty} k \big( a_{i,k} \, b_{j,k} - a_{j,k} \, b_{i,k} \big)
\end{align}
for $i, j \in \{1, \ldots, m\}$, see also \cite{MR1178485,MR1214374,Mil95}. Series
\eqref{LevyArea-Series-Rep} is the starting point for the construction of several
approximation algorithms. The main idea is to truncate the series in an appropriate
way. In order to simplify notation, we replace the random variables in 
\eqref{LevyArea-Series-Rep} by standard Gaussian random variables such that
\begin{align*}
	A_{(i,j)}(h) = \Delta W^i(h) \frac{\sqrt{h}}{\sqrt{2} \pi} \sum_{k=1}^{\infty}
	\frac{1}{k} X_{j,k} - \Delta W^j(h) \frac{\sqrt{h}}{\sqrt{2} \pi} \sum_{k=1}^{\infty}
	\frac{1}{k} X_{i,k}
	+ \frac{h}{2 \pi} \sum_{k=1}^{\infty} \frac{1}{k} \big( X_{i,k} \, 
	Y_{j,k} - X_{j,k} \, Y_{i,k} \big)
\end{align*}
where $a_{i,k} = \frac{\sqrt{h}}{\sqrt{2} \pi k} X_{i,k}$ and $b_{i,k} =
\frac{\sqrt{h}}{\sqrt{2} \pi k} Y_{i,k}$ with independent and identically $\Nd(0,1)$
distributed random variables $X_{i,k}$ and $Y_{i,k}$. For some 
$\Se \in \mathbb{N}$ define the truncated series
\begin{align} \label{A-ij-truncated}
	A_{(i,j)}^{(\Se)}(h) = \Delta W^i(h) \frac{\sqrt{h}}{\sqrt{2} \pi} \sum_{k=1}^{\Se}
	\frac{1}{k} X_{j,k} - \Delta W^j(h) \frac{\sqrt{h}}{\sqrt{2} \pi} \sum_{k=1}^{\Se}
	\frac{1}{k} X_{i,k}
	+ \frac{h}{2 \pi} \sum_{k=1}^{\Se} \frac{1}{k} \big( X_{i,k} \, 
	Y_{j,k} - X_{j,k} \, Y_{i,k} \big) 
\end{align}
and we split the remainder term into two parts denoted as 
\begin{align}
	\label{Reminder-R1-ij}
	R_{(i,j)}^{1,(\Se)}(h) &= \Delta W^i(h) \frac{\sqrt{h}}{\sqrt{2} \pi} 
	\sum_{k=\Se + 1}^{\infty}
	\frac{1}{k} X_{j,k} - \Delta W^j(h) \frac{\sqrt{h}}{\sqrt{2} \pi} 
	\sum_{k=\Se + 1}^{\infty} \frac{1}{k} X_{i,k} \, , \\
	\label{Reminder-R2-ij}
	R_{(i,j)}^{2,(\Se)}(h) &= \frac{h}{2 \pi} \sum_{k=\Se+1}^{\infty} \frac{1}{k} 
	\big( X_{i,k} \, Y_{j,k} - X_{j,k} \, Y_{i,k} \big)
\end{align}
such that $A_{(i,j)}(h) = A_{(i,j)}^{(\Se)}(h) + R_{(i,j)}^{1,(\Se)}(h) 
+ R_{(i,j)}^{2,(\Se)}(h)$ for all $i,j \in \{1, \ldots, m \}$.
\\ \\
An approximation for the L\'{e}vy area $A_{(i,j)}(h)$ can be calculated by simply
truncating the Fourier series after $n$ summands, i.e., calculating 
$A_{(i,j)}^{(\Se)}(h)$. However, as we will see in
Section~\ref{Sec:Vec-reduction-dim}, the conditional distribution of
$R_{(i,j)}^{1,(\Se)}(h)$ given the increments $\Delta W^i(h)$ and $\Delta W^j(h)$
is Gaussian. So, this term can be included for the approximation
as well resulting in some smaller error constant and this is exactly the algorithm 
that has been proposed by Milstein~\cite{Mil95} and it is also considered in
Kloeden, Platen and Wright~\cite{MR1178485} in a broader context, see also
\cite{MR1214374}. For $\Se \in \mathbb{N}$, let 
\begin{align} \label{Iij-Truncated-Fourier-Series-Alg}
	I_{(i,j)}^{FS,(\Se)}(h) = \frac{1}{2} \big( \Delta W^i(h) \Delta W^j(h) - h \, \ind_{i=j}
	\big) + A_{(i,j)}^{(\Se)}(h) + R_{(i,j)}^{1,(\Se)}(h)
\end{align}
denote the truncated Fourier series approximation for $I_{(i,j)}(h)$. Then, 
the sequence
$I_{(i,j)}^{FS,(\Se)}(h)$ converges in $L^2(\Omega)$-norm to $I_{(i,j)}(h)$ 
as $\Se \to \infty$ with
\begin{align} \label{Iij-Truncated-FS-Alg-error}
	\big( \Erw \big( \big| I_{(i,j)}(h) - I_{(i,j)}^{FS,(\Se)}(h) \big|^2 \big) \big)^{1/2} 
	= \bigg( \frac{h^2}{12} - \frac{h^2}{2 \pi^2} \sum_{k=1}^{\Se} \frac{1}{k^2} 
	\bigg)^{1/2}
	\leq \frac{h}{\pi \sqrt{2 \Se}}
\end{align}
for $i, j \in \{1, \ldots, m\}$ with $i \neq j$, see \cite{MR1178485,MR1214374,Mil95}.
The last estimate follows due to $\sum_{k=\Se +1}^{\infty} \frac{1}{k^2} \leq
\frac{1}{\Se}$.
As a result of this, for the discrete time approximation of the solutions of some
SDE with an strong order 1 method like the Milstein scheme \cite{MR1214374,Mil95}
or a stochastic Runge-Kutta method \cite{MR2669396} one has to choose 
$\Se \approx h^{-1}$ in order to preserve the strong order of convergence when 
replacing $I_{(i,j)}(h)$ by $I_{(i,j)}^{FS,(\Se)}(h)$ in the approximation scheme.
\section{Conditional distribution of the truncation terms}
\label{Sec:Vec-reduction-dim}
For the derivation of the approximation algorithm, it is sufficient to approximate
$A_{(i,j)}(h)$ for the case $j>i$ due to properties \eqref{Prop-Levy-Area-ij-ji}
and \eqref{Prop-Levy-Area-ii} of the L\'{e}vy areas. To begin with, we analyze 
the joint conditional distribution of the truncation terms $R_{(i,j)}^{1,(\Se)}(h)$
and $R_{(i,j)}^{2,(\Se)}(h)$ for $j>i$. 
Following the notation in \cite{MR3949104,MR1843055}, we first 
introduce some matrix operations that will be used in the following.
For $B = (b_{i,j}) \in \mathbb{R}^{m \times n}$ and $C \in 
\mathbb{R}^{p \times q}$, let $B \otimes C = (b_{i,j} \, C)$ denote the Kronecker
product of $B$ and $C$, which is in $\mathbb{R}^{mp \times nq}$. Further let 
$\vecop \colon \mathbb{R}^{m \times n} \to \mathbb{R}^{mn \times 1}$ denote
the vectorization operator such that
$\vecop(B) \in \mathbb{R}^{mn \times 1}$ is the vector 
that one obtains by stacking the columns of matrix $B$ one upon the other, i.e.,
\begin{align*}
	\vecop(B) &= \begin{pmatrix} b_{\cdot,1} \\ b_{\cdot,2} \\ \vdots \\ b_{\cdot,n}
		\end{pmatrix} \, .
\end{align*}
For a $m \times n$ matrix, the inverse operator of $\vecop$ is denoted as
$\matop_{m,n} \colon \mathbb{R}^{mn \times 1} \to \mathbb{R}^{m \times n}$ 
such that $\matop_{m,n}(\vecop(B)) = B$.
For simplification of notation, we make use of the permutation matrix $P_m \in 
\mathbb{R}^{m^2 \times m^2}$ defined as 
\begin{align*}
	P_m &= \sum_{i=1}^m \mathbf{e}_i^{\Tt} \otimes (I_m \otimes \mathbf{e}_i) \, , 
\end{align*}
where $\mathbf{e}_i \in \mathbb{R}^m$ denotes the $i$-th unit vector and $I_m$ 
is the $m \times m$ identity matrix. Then, it holds $P_m = P_m^{\Tt}$ and
$P_m \, P_m^{\Tt} = I_{m^2}$, see also
\cite{MR3949104,MR520247,MR1843055}. Moreover, $P_m (u \otimes v) = 
v \otimes u$, $P_m ( B \otimes u) = u \otimes B$ and $P_m (B \otimes C) P_m =
C \otimes B$ for $u, v \in \mathbb{R}^m$ and $B, C \in \mathbb{R}^{m \times m}$, 
see \cite[Theorem~3.1]{MR520247}. 
Due to relation \eqref{Prop-Levy-Area-ij-ji}, it is sufficient to approximate
$A_{(i,j)}(h)$ for $i<j$ and we denote
\begin{align*}
	 \hat{A}(h) &= \big(A_{(1,2)}(h), \ldots, A_{(1,m)}(h), A_{(2,3)}(h), \ldots,
	 A_{(2,m)}(h), \ldots, A_{(l,l+1)}(h), \ldots, \\*	
	 & \quad \quad A_{(l,m)}(h), \ldots, A_{(m-1,m)}(h) \big)^{\Tt} ,
\end{align*}
which is a vector of length $M = \frac{m(m-1)}{2}$.
Next, we introduce the selection matrix $H_m$ of size $M \times m^2$ defined
as
\begin{equation} \label{SelectionMatrix}
	H_m = \begin{pmatrix}
		0_{m-1\times 1} & I_{m-1} & 0_{m-1\times m(m-1)}\\
		0_{m-2\times m+2} & I_{m-2} & 0_{m-2\times m(m-2)}\\
		\vdots   & \vdots & \vdots\\
		0_{m-l\times(l-1)m+l} & I_{m-l} & 0_{m-l\times m(m-l)}\\
		\vdots   & \vdots & \vdots\\
		0_{1\times(m-2)m+m-1} & 1 & 0_{1\times m}
	\end{pmatrix} \, .
\end{equation}
Then, it holds $H_m H_m^{\Tt} = I_M$ and $H_m P_m H_m^{\Tt} = 0_{M 
\times M}$, see \cite{MR1843055}.
If $B$ is some $m \times m$ matrix, then $H_m$ applied to $\vecop(B)$ picks out
exactly those $M$ elements of $\vecop(B)$ that belong to the lower triangle matrix
of $B$, i.e., the elements $B_{i,j}$ with $i>j$, and gives them back in a vector
of length $M$ in the same order as they are given in the vector $\vecop(B)$.
Therefore, it holds 
\begin{align*}
	\hat{A}(h) = H_m \vecop \big( A(h)^{\Tt} \big) \, ,
\end{align*}
which is a vector of length $M$.
Analogously, we can write \eqref{A-ij-truncated} as well as the remainder terms
\eqref{Reminder-R1-ij} and \eqref{Reminder-R2-ij} as 
\begin{align}
	\hat{A}^{(\Se)}(h) &= H_m \vecop \big( {A^{(\Se)}(h)}^{\Tt} \big) \nonumber \\
	\label{A-tilde-n-Vec-Mat}
	&= \frac{\sqrt{h}}{\sqrt{2} \pi} \sum_{k=1}^{\Se} \frac{1}{k} H_m (I_{m^2}-P_m)
	( \Delta W(h) \otimes X_k ) 
	+ \frac{h}{2 \pi} \sum_{k=1}^{\Se} \frac{1}{k} H_m (P_m-I_{m^2})
	( Y_k \otimes X_k ) \\
	\label{R1n-Vec-Mat}
	\hat{R}^{1,(\Se)}(h) &= H_m \vecop \big( {R^{1,(\Se)}}^{\Tt} \big) 
	= \frac{\sqrt{h}}{\sqrt{2} \pi} \sum_{k=\Se +1}^{\infty} \frac{1}{k} H_m (I_{m^2}-P_m)
	( \Delta W(h) \otimes X_k ) \\
	\label{R2n-Vec-Mat}
	\hat{R}^{2,(\Se)}(h) &= H_m \vecop \big( {R^{2,(\Se)}}^{\Tt} \big)
	= \frac{h}{2 \pi} \sum_{k=\Se +1}^{\infty} \frac{1}{k} H_m (P_m-I_{m^2})
	( Y_k \otimes X_k )
\end{align}
with $X_k = (X_{1,k}, \ldots, X_{m,k})^{\Tt}$ and $Y_k = (Y_{1,k}, \ldots,
Y_{m,k})^{\Tt}$.
Here, 
$X_k \sim \Nd(0_m, I_m)$, $Y_k \sim \Nd(0_m, I_m)$ for $k \in \mathbb{N}$ 
and $\Delta W(h) \sim \Nd(0_m, h I_m)$ are all independent Gaussian random
variables. Thus, given $\Delta W(h)$, the random vector $\hat{A}^{(\Se)}(h)$
is conditionally independent from $\hat{R}^{1,(\Se)}(h)$ and $\hat{R}^{2,(\Se)}(h)$.
Further, it holds $\hat{A}(h) = \hat{A}^{(\Se)}(h) + \hat{R}^{1,(\Se)}(h) 
+ \hat{R}^{2,(\Se)}(h)$.
\\ \\
Firstly, we consider the truncation term $\hat{R}^{1,(\Se)}(h)$. The series 
$\sum_{k=\Se +1}^{\infty} \frac{1}{k} X_k$ converges $\Prob$-a.s., see e.g.\
\cite[Theorem~7.5]{MR2767184}. Since the random variables 
$X_{1,k}, \ldots X_{m,k}$ for $k \in \mathbb{N}$ are all i.i.d.\ Gaussian random
variables, it follows that $\sum_{k=\Se +1}^{\infty} \frac{1}{k} X_k$ 
is Gaussian with expectation $0_m$ and covariance matrix 
\begin{align*}
	\Sigma^{1,(\Se)} = \bigg( \frac{\pi^2}{6} - \sum_{k=1}^{\Se} \frac{1}{k^2} \bigg) 
	I_m \, ,
\end{align*}
see also \cite{MR1178485,MR1214374,Mil95}. Thus, we have
\begin{align*}
	\sum_{k=\Se +1}^{\infty} \frac{1}{k} X_k = \big( \Sigma^{1,(\Se)} \big)^{1/2}
	\Psi^{1,(\Se)}
\end{align*}
where $\Psi^{1,(\Se)}$ is a $\Nd(0_m, I_m)$ distributed random vector
given by
\begin{align} \label{RV-Psi-1-n-Exact}
	\Psi^{1,(\Se)} = \big( \Sigma^{1,(\Se)} \big)^{-1/2} 
	\sum_{k=\Se +1}^{\infty} \frac{1}{k} X_k \, .
\end{align}
Therefore, it follows that we can rewrite the truncation term 
$\hat{R}^{1,(\Se)}(h)$ as
\begin{align} \label{Def-hat-A1n}
	\hat{A}^{1,(\Se)}(h) &= \frac{\sqrt{h}}{\sqrt{2} \pi} \bigg( \frac{\pi^2}{6} - \sum_{k=1}^{\Se} \frac{1}{k^2} \bigg)^{1/2} H_m \, (I_{m^2}-P_m) \,
	( \Delta W(h) \otimes \Psi^{1,(\Se)} ) \, .
\end{align}
%
Next, we analyze the second truncation term $\hat{R}^{2,(\Se)}(h)$. We proceed 
in a similar way as in \cite{MR1843055}, however differing in some crucial
points. First of all, for each $k \in \mathbb{N}$ and given the random vector $X_k$ 
one observes that the summand $(P_m - I_{m^2}) (Y_k \otimes X_k)$ is 
conditionally Gaussian distributed with conditional mean 
\begin{align*}
	\Erw( (P_m - I_{m^2}) (Y_k \otimes X_k) \mid X_k) 
	&= (P_m - I_{m^2}) ( \Erw(Y_k) \otimes X_k ) = 0_{m^2}
\end{align*}
and conditional covariance matrix $\Sigma(X_k)$ that is given by
\begin{align*}
	\Sigma(X_k) &= \Erw \big( (P_m - I_{m^2}) (Y_k \otimes X_k) \big( 
	(P_m - I_{m^2}) (Y_k \otimes X_k) \big)^{\Tt} \bigMid X_k \big) \\
	&= (P_m - I_{m^2}) \big( \Erw( Y_k Y_k^{\Tt} \Mid X_k) 
	\otimes (X_k X_k^{\Tt}) \big) (P_m - I_{m^2})^{\Tt} \\
	&= (P_m - I_{m^2}) ( I_{m} \otimes (X_k X_k^{\Tt})) (P_m - I_{m^2})^{\Tt} \, .
\end{align*}
Thus, given $X^{(\Se)} = (X_k)_{k \geq \Se+1}$ the conditional distribution of 
$\hat{R}^{2,(\Se)}$ is Gaussian with conditional mean $0_M$ and conditional
covariance matrix 
\begin{align} \label{Sigma-2-n-Def-Eqn}
	\Sigma^{2,(\Se)}(X^{(\Se)}) &= \frac{h^2}{4 \pi^2} \sum_{k=\Se +1}^{\infty}
	\frac{1}{k^2} H_m \Sigma(X_k) H_m^{\Tt}
\end{align}
since the summands in \eqref{R2n-Vec-Mat} are independent. 
As a result of this, it holds
\begin{align*}
	\hat{R}^{2,(\Se)}(h) = \frac{h}{2 \pi} \bigg( \sum_{k=\Se +1}^{\infty} \frac{1}{k^2} 
	H_m \Sigma(X_k) H_m^{\Tt} \bigg)^{1/2} \Psi^{2,(\Se)}
\end{align*}
with some $\Nd( 0_M, I_M)$ distributed random vector $\Psi^{2,(\Se)}$ that 
is given by
\begin{align} \label{RV-Psi-2-n-Exact}
	\Psi^{2,(\Se)} = \frac{2 \pi}{h} \bigg( \sum_{k=\Se +1}^{\infty} \frac{1}{k^2} 
	H_m \Sigma(X_k) H_m^{\Tt} \bigg)^{-1/2} \hat{R}^{2,(\Se)}(h) \, .
\end{align}
An explicit representation of the conditional covariance matrix $H_m \Sigma(X_k)
H_m^{\Tt}$ is derived in Appendix~\ref{Appendix-A}.
Note that for each $r,s \in \{1, \ldots, M\}$ the series $\frac{h^2}{4 \pi^2}
\sum_{k=\Se +1}^{\infty} \frac{1}{k^2} (H_m \Sigma(X_k) H_m^{\Tt})_{r,s}$ in
\eqref{Sigma-2-n-Def-Eqn} converges absolutely in $L^p(\Omega)$-norm for 
$p \geq 1$ and thus it also converges $\Prob$-a.s.\ \cite[Prop.~8.7]{MR3618289}. 
\\ \\
These considerations build the basis for the construction of our approximation
algorithm for iterated stochastic integrals. 
Note that $\Psi^{1,(\Se)}$ depends only on the random variables $X_k$ for
$k \geq n+1$, while $\Psi^{2,(\Se)}$ is independent from the random variables
$X_k$. Therefore, $\Psi^{1,(\Se)}$ and $\Psi^{2,(\Se)}$ are stochastically 
independent, which is essential for our algorithm.
However, in contrast to truncation term 
$\hat{R}^{1,(\Se)}(h)$ we do not know the conditional distribution of
$\hat{R}^{2,(\Se)}(h)$ completely as it still depends on the random variables 
$X_k$ for $k \geq \Se+1$.
\section{Approximation algorithm for iterated stochastic integrals and 
error estimates}
\label{Sec:Alg}
In order to approximate truncation term $\hat{R}^{2,(\Se)}(h)$, we define the 
$M \times M$ matrix $\Sigma^{2,\infty}$ as the mean of the conditional covariance
matrix of $\hat{R}^{2,(\Se)}(h)$. Due to Lebesgue's dominated convergence
theorem (see, e.g., \cite[Theorem~9.2]{MR1956867}) together with
Appendix~\ref{Appendix-A}, we have
\begin{align}
	\Sigma^{2, \infty} &= \Erw \big( \Sigma^{2,(\Se)}(X^{(\Se)}) \big) 
	\nonumber \\
	&= \frac{h^2}{4 \pi^2} \sum_{k=\Se +1}^{\infty} \frac{1}{k^2} \,
	\Erw \big( H_m \Sigma(X_1) H_m^{\Tt} \big) 
	\nonumber \\
	&= \frac{h^2}{4 \pi^2} \sum_{k=\Se +1}^{\infty} \frac{1}{k^2} \,
	H_m (P_m - I_{m^2}) \big( I_m \otimes \Erw(X_1 X_1^{\Tt}) \big) 
	(P_m - I_{m^2})^{\Tt} H_m^{\Tt} 
	\nonumber \\
	&= \frac{h^2}{4 \pi^2} \sum_{k=\Se +1}^{\infty} \frac{1}{k^2} \, 
	H_m (P_m - I_{m^2}) (P_m - I_{m^2})^{\Tt} H_m^{\Tt} 
	\nonumber \\
	\label{Sigma-Infty}
	&= \frac{h^2}{4 \pi^2} \sum_{k=\Se +1}^{\infty} \frac{1}{k^2} \,
	2 I_M \, .
\end{align}
The matrix $\Sigma^{2, \infty}$ serves as an approximation of the conditional
covariance matrix $\Sigma^{2,(\Se)}(X^{(\Se)})$. Therefore, we define
\begin{align} \label{Def-hat-A2n}
	\hat{A}^{2,(\Se)}(h) &= \big( \Sigma^{2, \infty} \big)^{1/2} \Psi^{2,(n)}
	= \frac{h}{\sqrt{2} \pi} \bigg( \frac{\pi^2}{6} -
	\sum_{k=1}^{\Se} \frac{1}{k^2} \bigg)^{1/2} \Psi^{2,(n)}
\end{align}
as an approximation of the truncation term $\hat{R}^{2,(\Se)}(h)$, where
$\Psi^{2,(\Se)}$ is the $\Nd(0_M, I_M)$ distributed random vector in 
\eqref{RV-Psi-2-n-Exact}. 
\\ \\
Now, we can define the approximation of the iterated stochastic integrals
as follows: For the $m \times m$ matrix $I(h) = ( I_{(i,j)}(h) )_{1 \leq i,j \leq m}$
of all iterated stochastic integrals it holds 
\begin{align*}
	\vecop \big( I(h)^{\Tt} \big) &= \frac{1}{2} \big( \Delta W(h) \otimes \Delta W(h) 
	- h \vecop(I_m) \big) + \vecop \big( A(h)^{\Tt} \big) \, .
\end{align*}
Let
\begin{align*}
	\hat{I}(h) &= H_m \vecop \big( I(h)^{\Tt} \big) \\
	&= \frac{1}{2} H_m \big( \Delta W(h) \otimes \Delta W(h) 
	- h \vecop(I_m) \big) + \hat{A}(h) \, .
\end{align*}
Now, we approximate the $M$-dimensional vector $\hat{I}(h)$ containing 
all iterated stochastic integrals $I_{(i,j)}(h)$ for $j > i$ by the newly proposed
algorithm that is defined by the approximation
\begin{align} \label{Approximation-vec-Version-hat-Iij}
	\hat{I}^{(\Se)}(h) = \frac{1}{2} H_m \big( \Delta W(h) \otimes \Delta W(h) 
	- h \vecop(I_m) \big) + \hat{A}^{(\Se)}(h) + \hat{A}^{1,(\Se)}(h)
	+ \hat{A}^{2,(\Se)}(h)
\end{align}
with $\hat{A}^{(\Se)}(h)$, $\hat{A}^{1,(\Se)}(h)$ and $\hat{A}^{2,(\Se)}(h)$ 
defined in \eqref{A-tilde-n-Vec-Mat}, \eqref{Def-hat-A1n} and 
\eqref{Def-hat-A2n}, respectively.
\\ \\
Now, from \eqref{Approximation-vec-Version-hat-Iij} one can rebuild 
the full $m^2$-dimensional approximation vector by
\begin{align*}
	\vecop \big( I^{(\Se)}(h)^{\Tt} \big) &= \frac{1}{2} \big( \Delta W(h) 
	\otimes \Delta W(h) - h \vecop(I_m) \big) \\
	&\quad + ( I_{m^2} - P_m ) H_m^{\Tt} \big( 
	\hat{A}^{(\Se)}(h) + \hat{A}^{1,(\Se)}(h)
	+ \hat{A}^{2,(\Se)}(h) \big)
\end{align*}
and with $I^{(\Se)}(h) = ( \matop_{m,m} ( \vecop ( I^{(\Se)}(h)^{\Tt} ) ) )^{\Tt}$ 
one obtains the $m \times m$ matrix containing all approximations
$I_{(i,j)}^{(\Se)}$ for $1 \leq i,j \leq m$, which is an approximation of the 
matrix $I(h)$. The approximation $I^{(\Se)}(h)$ is unbiased, i.e., it holds:
\begin{prop} \label{Prop-Mean-error-unbiased}
	For any $\Se \in \mathbb{N}$ it holds for the approximations
	$I_{(i,j)}^{(\Se)}(h)$ of the iterated stochastic integral $I_{(i,j)}(h)$ that
	\begin{equation} \label{Prop-Mean-error-unbiased-eqn}
		\big| \Erw \big( I_{(i,j)}(h) - I_{(i,j)}^{(\Se)}(h) \big) \big| = 0
	\end{equation}
	for all $1 \leq i,j \leq m$ and $h>0$.
\end{prop}
\begin{proof}
	Let $\Se \in \mathbb{N}$ and $h>0$ arbitrarily fixed. For $i=j$ the assertion 
	obviously holds because $I_{(i,i)}(h) = I_{(i,i)}^{(\Se)}(h)$. Therefore, it is
	sufficient to consider the case $i \neq j$. Since $\hat{A}(h) = \hat{A}^{(n)}(h) 
	+ \hat{R}^{1,(\Se)}(h) + \hat{R}^{2,(\Se)}(h)$, it holds
	\begin{align*}
		&\Erw \big( \hat{I}(h) - \hat{I}^{(\Se)}(h) \big) \\
		&= \Erw \big( \hat{A}(h) - \hat{A}^{(\Se)}(h) - \hat{A}^{1,(\Se)}(h)
		- \hat{A}^{2,(\Se)}(h) \big) \\
		&= \Erw \big( \hat{R}^{1,(\Se)}(h) + \hat{R}^{2,(\Se)}(h)
		- \hat{A}^{1,(\Se)}(h) - \hat{A}^{2,(\Se)}(h) \big) \\
		&= \Erw \bigg( \frac{h}{2 \pi} \bigg( \sum_{k=\Se +1}^{\infty} \frac{1}{k^2} 
		H_m \Sigma(X_k) H_m^{\Tt} \bigg)^{1/2} \Psi^{2,(\Se)}
		- \frac{h}{\sqrt{2} \pi} \bigg( \frac{\pi^2}{6} -
		\sum_{k=1}^{\Se} \frac{1}{k^2} \bigg)^{1/2} \Psi^{2,(n)} \bigg) \\
		&= 0
	\end{align*}
	because $\hat{R}^{1,(\Se)}(h) = \hat{A}^{1,(\Se)}(h)$ $\Prob$-a.s.\ and
	since $\Psi^{2,(n)}$ is $\Nd(0_M,I_M)$ distributed.
\end{proof}
\subsection{Error estimates in $L^2$-norm}
In order to analyze the error of the approximation $\hat{I}^{(\Se)}(h)$ given 
by the proposed algorithm in \eqref{Approximation-vec-Version-hat-Iij} subject 
to the parameters $h$ and $\Se \in \mathbb{N}$, we need some auxiliary 
results first.
\begin{prop} \label{Prop-dist-SigmaN-SigmaInfty}
	The sequence of conditional covariance matrices 
	$( \Sigma^{2,(\Se)}(X^{(\Se)}) )_{n \in \mathbb{N}}$ converges to the matrix 
	$\Sigma^{2, \infty}$ in $L^2(\Omega, \| \cdot \|_F)$-norm as $\Se \to \infty$. 
	Further, it holds:
	\begin{enumerate}[(i)]
		\item \label{Prop-dist-SigmaN-SigmaInfty-i}
		For any $\Se \in \mathbb{N}$ and $r \in \{1, \ldots, M\}$ it holds
		\begin{equation} \label{Prop-dist-SigmaN-SigmaInfty-eqn1}
			\sum_{q=1}^M \Erw \big( \big| \big( \Sigma^{2,(\Se)}(X^{(\Se)}) 
			- \Sigma^{2, \infty} \big)_{r,q} \big|^2 \big) = 
			\frac{h^4 m}{8 \pi^4} \sum_{k=\Se+1}^{\infty} k^{-4}
			\, .
		\end{equation}
		\item \label{Prop-dist-SigmaN-SigmaInfty-ii}
		For any $\Se \in \mathbb{N}$ it holds
		\begin{equation} \label{Prop-dist-SigmaN-SigmaInfty-eqn2}
			\Erw \big( \big\| \Sigma^{2,(\Se)}(X^{(\Se)}) - \Sigma^{2, \infty} 
			\big\|_F^2 \big) 
			= \frac{h^4 m^2 (m-1)}{16 \pi^4} \sum_{k=\Se+1}^{\infty} k^{-4}
			\, .
		\end{equation}
	\end{enumerate}
\end{prop}
\begin{proof}
	Due to \eqref{Sigma-Infty}, it holds $\Erw( \Sigma^{2,(\Se)}(X^{(\Se)}) ) = 
	\Sigma^{2, \infty}$. Let $r \in \{1, \ldots, M\}$ be arbitrarily fixed.
	Then, it follows with Lebesgue's dominated convergence theorem 
	(see, e.g., \cite[Theorem~9.2]{MR1956867}) together
	with Appendix~\ref{Appendix-A} that
	\begin{align*}
		\sum_{q=1}^M \Erw \big( \big| \big( \Sigma^{2,(\Se)}(X^{(\Se)}) 
		- \Sigma^{2, \infty} \big)_{r,q} \big|^2 \big)
		&= \sum_{q=1}^M \Erw \big( \big| \Sigma_{r,q}^{2,(\Se)}(X^{(\Se)}) 
		- \Sigma_{r,q}^{2, \infty} \big|^2 \big) \\
		&= \sum_{q=1}^M \Erw \big( \big| \Sigma_{r,q}^{2,(\Se)}(X^{(\Se)}) 
		- \Erw( \Sigma_{r,q}^{2,(\Se)}(X^{(\Se)}) ) \big|^2 \big) \\
		&= \sum_{q=1}^M \Var \big( \Sigma_{r,q}^{2,(\Se)}(X^{(\Se)}) \big) \\
		&= \frac{h^4}{16 \pi^4} \sum_{q=1}^M \sum_{k=\Se+1}^{\infty} \frac{1}{k^4} 
		\Var \Big( \big( H_m \Sigma(X_1) H_m^{\Tt} \big)_{r,q} \Big)
	\end{align*}
	since $H_m \Sigma(X_k) H_m^{\Tt}$ are i.i.d.\ for $k \geq \Se+1$. Calculating
	the symmetric $M \times M$ matrix $H_m \Sigma(X_1) H_m^{\Tt}$ explicitly, 
	see Appendix~\ref{Appendix-A}, it follows that each row as well as each 
	column has
	exactly $2m-4$ entries of type $\pm X_{i,1} X_{j,1}$, each of them with 
	different indices $i \neq j$, 
	and one entry on the diagonal position of type $X_{i,1}^2 + X_{j,1}^2$ for
	some $i \neq j$. The remaining $M-2m-3$ entries are zero. Therefore, we get
	\begin{align*}
		\sum_{q=1}^M \Var \Big( \big( H_m \Sigma(X_1) H_m^{\Tt} \big)_{r,q} \Big) 
		&= \Var \big( X_{1,1}^2 + X_{2,1}^2 \big) + (2m-4) \Var 
		\big( X_{1,1} X_{2,1} \big) \\
		&= \Erw \big( \big( X_{1,1}^2 + X_{2,1}^2 \big)^2 \big) 
		- \big( \Erw \big( X_{1,1}^2 + X_{2,1}^2 \big) \big)^2 \\
		&\quad + (2m-4) \big( \Erw \big( X_{1,1}^2 X_{2,1}^2 \big) 
		- \big( \Erw \big( X_{1,1} X_{2,1} \big) \big)^2 \big) \\
		&= 2 m \, ,
	\end{align*}
	and (i) follows directly. Moreover, applying (i) we calculate
	\begin{align*}
		\Erw \big( \big\| \Sigma^{2,(\Se)}(X^{(\Se)}) - \Sigma^{2, \infty} \big\|_F^2 \big)
		&= \sum_{r=1}^M \sum_{q=1}^M \Erw \big( \big| \big(
		\Sigma^{2,(\Se)}(X^{(\Se)}) - \Sigma^{2, \infty} \big)_{r,q} \big|^2 \big) \\
		&= \frac{h^4}{16 \pi^4} M \sum_{k=\Se+1}^{\infty} \frac{1}{k^4} \, 2m
	\end{align*}
	which completes the proof of statement (ii).
\end{proof}
\noindent
The following auxiliary Lemma will be useful for the proof of error estimates for the
proposed algorithm. A similar result under slightly different assumptions can be
found in \cite{PleisDiss20,MR1843055}.
\begin{lem} \label{Lem-Sqrt-A-B-estimate}
	Let $A,  B \in \mathbb{R}^{q \times q}$ be symmetric commuting matrices, i.e., $AB=BA$, such that $A$ is  positive semi-definite and $B$ is
	positive definite. Further, let $C \in \mathbb{R}^{r \times q}$ and let 
	$\lambda_{\min}^B >0$ denote the smallest eigenvalue of $B$. Then, it holds:
	\begin{enumerate}[(i)]
		\item \label{Lem-Sqrt-A-B-estimate-eq-2-2}
			$\displaystyle \| C (A^{1/2} - B^{1/2} ) \|_2 
			\leq \tfrac{1}{\sqrt{\lambda_{\min}^B}} \| C (A-B) \|_2$
		\item \label{Lem-Sqrt-A-B-estimate-eq-2-F}
			$\displaystyle \| C (A^{1/2} - B^{1/2} ) \|_2 
			\leq \tfrac{1}{\sqrt{\lambda_{\min}^B}} \| C (A-B) \|_F$
		\item \label{Lem-Sqrt-A-B-estimate-eq-F-F}
			$\displaystyle \| C (A^{1/2} - B^{1/2} ) \|_F 
			\leq \tfrac{1}{\sqrt{\lambda_{\min}^B}} \| C (A-B) \|_F$
	\end{enumerate}
\end{lem}
\begin{proof}
	By the assumptions of Lemma~\ref{Lem-Sqrt-A-B-estimate} it follows that
	$A$ and $B$ are normal and that $A^{1/2}+B^{1/2}$ is symmetric and positive
	definite. Moreover, it follows that $A^{1/2} B^{1/2} = B^{1/2} A^{1/2}$. 
	Then, making use of the fact that the applied matrix norm is sub-multiplicative, we get
	\begin{equation} \label{Proof-Lem-Sqrt-A-B-estimate-2-Norm-Fak}
		\begin{split}
		\| C ( A^{1/2} - B^{1/2}) \|_2 &= \| C ( A^{1/2} - B^{1/2}) ( A^{1/2} + B^{1/2})
		( A^{1/2} + B^{1/2})^{-1} \|_2 \\
		&= \| C (A-B) ( A^{1/2} + B^{1/2})^{-1} \|_2 \\
		&\leq \| C (A-B) \|_2 \, \| ( A^{1/2} + B^{1/2})^{-1} \|_2 \, .
		\end{split}
	\end{equation}
	Since $A$ and $B$ are symmetric and commute, they are both simultaneously
	diagonalizable, i.e., there exist a unitary matrix $U$ such that 
	$A=U D_A U^{\Tt}$ and $B= U D_B U^{\Tt}$ where $D_A 
	= \diag ( \lambda_1^A, \ldots, \lambda_q^A)$
	and $D_B = \diag ( \lambda_1^B, \ldots, \lambda_q^B)$ with eigenvalues 
	$\lambda_l^A$ of $A$ and $\lambda_l^B$ of $B$ for $1 \leq l \leq q$.
	Then, it follows
	\begin{equation} \label{Proof-Lem-Sqrt-A-B-estimate-2-Norm-EV}
		\begin{split}
		\| ( A^{1/2} + B^{1/2})^{-1} \|_2 &= \| (U (D_A^{1/2} + D_B^{1/2}) 
		U^{\Tt} )^{-1} \|_2 \\
		&= \| U (D_A^{1/2} + D_B^{1/2})^{-1} U^{\Tt} \|_2 \\
		&= \bigg\| \diag \bigg( \frac{1}{\sqrt{\lambda_1^A} + \sqrt{\lambda_1^B}}, 
		\ldots, \frac{1}{\sqrt{\lambda_q^A} + \sqrt{\lambda_q^B}} \bigg) \bigg\|_2 \\
		&= \max_{1 \leq k \leq q} \frac{1}{\sqrt{\lambda_k^A} + \sqrt{\lambda_k^B}}
		\leq \frac{1}{\sqrt{\lambda_{\min}^B}} \, .
		\end{split}
	\end{equation}
	Then, (i) follows from \eqref{Proof-Lem-Sqrt-A-B-estimate-2-Norm-Fak} and
	\eqref{Proof-Lem-Sqrt-A-B-estimate-2-Norm-EV}. Because $\| H \|_2 \leq 
	\| H \|_F$ for any $H \in \mathbb{R}^{q \times q}$, (ii) follows from (i).
	Finally, it holds
	\begin{equation} \label{Proof-Lem-Sqrt-A-B-estimate-2-Norm-Fa-F2}
		\begin{split}
		\| C ( A^{1/2} - B^{1/2}) \|_F &= \| C ( A^{1/2} - B^{1/2}) ( A^{1/2} + B^{1/2})
		( A^{1/2} + B^{1/2})^{-1} \|_F \\
		&= \| (( A^{1/2} + B^{1/2})^{-1})^{\Tt} (A-B)^{\Tt} C^{\Tt} \|_F \\
		&\leq \| ( A^{1/2} + B^{1/2})^{-1} \|_2 \, \| C (A-B) \|_F \, ,
		\end{split}
	\end{equation}
	and (iii) follows from \eqref{Proof-Lem-Sqrt-A-B-estimate-2-Norm-Fa-F2}
	and \eqref{Proof-Lem-Sqrt-A-B-estimate-2-Norm-EV}.
\end{proof}
\noindent
We consider the error for the approximation of the iterated stochastic integrals
calculated by the proposed algorithm
w.r.t.\ the $L^p(\Omega)$-norm for $p \geq 2$. However, 
we treat the case $p=2$ separately because we can obtain a sharper upper bound
for this case compared to the general case due to a more sophisticated proof.
\begin{thm} \label{Theo-I-In-Fehler-Estimate}
	Let $\Se \in \mathbb{N}$ and let ${I}_{(i,j)}^{(\Se)}(h)$
	be the approximation given by \eqref{Approximation-vec-Version-hat-Iij} 
	of the iterated stochastic integral $I_{(i,j)}(h)$
	for $1 \leq i, j \leq m$. Further, let $I(h) = (I_{(i,j)}(h))_{1 \leq i,j \leq m}$
	and $I^{(\Se)}(h) = ( I_{(i,j)}^{(\Se)}(h) )_{1 \leq i,j \leq m}$ denote the
	corresponding $m \times m$ matrices. Then, it holds:
	\begin{enumerate}[(i)]
		\item
		\begin{equation} \label{Theo-I-In-Fehler-Estimate-eqn1}
			\max_{1 \leq i,j \leq m} \Erw \big( | I_{(i,j)}(h) - I_{(i,j)}^{(\Se)}(h) |^2 \big)
			\leq \frac{h^2 m}{4 \pi^2} \, \frac{\sum_{k=\Se+1}^{\infty} k^{-4}}{
			\sum_{k=\Se+1}^{\infty} k^{-2} } 
		\end{equation}
		\item
		\begin{equation} \label{Theo-I-In-Fehler-Estimate-eqn2}
			\Erw \big( \| I(h) - I^{(\Se)}(h) \|_F^2 \big)
			\leq \frac{h^2 m^2 (m-1)}{4 \pi^2} \, \frac{\sum_{k=\Se+1}^{\infty} k^{-4}}{
			\sum_{k=\Se+1}^{\infty} k^{-2} } 
		\end{equation}
	\end{enumerate}
\end{thm}
\begin{proof}
	Let $\Se \in \mathbb{N}$. Because $I_{(i,i)}(h) = I_{(i,i)}^{(\Se)}(h)$ 
	for all $i \in \{1, \ldots, m\}$ and with \eqref{Prop-Levy-Area-ij-ji}, it holds
	\begin{align*}
		\max_{1 \leq i,j \leq m} \Erw \big( | I_{(i,j)}(h) - I_{(i,j)}^{(\Se)}(h) |^2 \big)
		&= \max_{1 \leq r \leq M} \Erw \big( \big| \big( 
		\hat{A}(h) - \hat{A}^{(\Se)}(h)
		- \hat{A}^{1,(\Se)}(h) - \hat{A}^{2,(\Se)}(h) \big)_r \big|^2 \big) \\ 
		&= \max_{1 \leq r \leq M} \Erw \big( \big| \big( 
		\hat{R}^{2,(\Se)}(h) - \hat{A}^{2,(\Se)}(h) \big)_r \big|^2 \big) \\
		&= \max_{1 \leq r \leq M} \Erw \big( \big| \big(
		({\Sigma^{2,(\Se)}(X^{(\Se)})})^{1/2} \Psi^{2,(\Se)}
		- ({\Sigma^{2,\infty}})^{1/2} \Psi^{2,(\Se)} 
		\big)_r \big|^2 \big) \, .
	\end{align*}
	Let $r \in \{1, \ldots, M\}$ be arbitrarily fixed. To simplify notation, let $C(r) \in 
	\mathbb{R}^{1 \times M}$ with $C_{1,r}(r)=1$ and $C_{1,l}(r)=0$ for $l \neq r$.
	Further, denote $S(X^{(\Se)}) = ({\Sigma^{2,(\Se)}(X^{(\Se)})})^{1/2} 
	- ({\Sigma^{2,\infty}})^{1/2}$. 
	%
	We define
	\begin{align*}
		\varphi_r \big( X^{(\Se)}, \Psi^{2,(\Se)} \big) = \sum_{s=1}^M \big(
		({\Sigma^{2,(\Se)}(X^{(n)})})^{1/2} - ({\Sigma^{2,\infty}})^{1/2} \big)_{r,s} 
		\Psi^{2,(\Se)}_s
	\end{align*}
	and we first prove $\big| \varphi_r \big( X^{(\Se)}, \Psi^{2,(\Se)} \big) \big|^2 
	\in L^1(\Omega)$. Since $\Sigma^{2,\infty}$ and $\Sigma^{2,(\Se)}(X^{(\Se)})$ commute, we can apply Lemma~\ref{Lem-Sqrt-A-B-estimate} and get
	\begin{align*}
		\Erw \big( \big| \varphi_r \big( X^{(\Se)}, \Psi^{2,(\Se)} \big) \big|^2 \big) 
		&= \Erw \big( \big\| C(r) \, S(X^{(\Se)}) \, \Psi^{2,(\Se)} \big\|_2^{2} \big) \\
		&\leq \Erw \big( \big\| C(r) \, S(X^{(\Se)}) \big\|_F^2 \big) \, 
		\Erw \big( \big\| \Psi^{2,(\Se)} \big\|_2^{2} \big) \\
		&\leq \big( \lambda_{\min}^{2,\infty} \big)^{-1} \,
		\Erw \big( \big\| C(r) \, \big( \Sigma^{2,(\Se)}(X^{(\Se)})
		- \Sigma^{2,\infty} \big) \big\|_F^{2} \big) \, 
		\Erw \big( \big\| \Psi^{2,(\Se)} \big\|_2^{2} \big) \\
		&= M \, \big( \lambda_{\min}^{2,\infty} \big)^{-1} \,
		\sum_{s=1}^M \Erw \Big( \big( \Sigma^{2,(\Se)}(X^{(\Se)})
		- \Sigma^{2,\infty} \big)_{r,s}^2 \Big)
	\end{align*}
	where $\lambda_{\min}^{2,\infty} = \frac{h^2}{2 \pi^2} \sum_{k=\Se+1}^{\infty}
	\frac{1}{k^2}$ is the minimal eigenvalue of $\Sigma^{2,\infty}$. 
	Now, by applying
	Proposition~\ref{Prop-dist-SigmaN-SigmaInfty}~(\ref{Prop-dist-SigmaN-SigmaInfty-i}) 
	we get that $\Erw \big( \big| \varphi_r \big( X^{(\Se)}, \Psi^{2,(\Se)} \big) \big|^2
	\big)$ is uniformly bounded for all $r \in \{1, \ldots, M\}$ and all $n \in
	\mathbb{N}_0$. Thus, we can take the conditional expectation of 
	$\big| \varphi_r \big( X^{(\Se)}, \Psi^{2,(\Se)} \big) \big|^2$ w.r.t.\ $X^{(\Se)}$.
	%
	%
	Since $\Psi^{2,(\Se)}$ is independent of $X^{(\Se)}$, it follows
	\begin{align*}
		&\Erw \big( \big| \big(
		({\Sigma^{2,(\Se)}(X^{(\Se)})})^{1/2} \Psi^{2,(\Se)}
		- ({\Sigma^{2,\infty}})^{1/2} \Psi^{2,(n)} \big)_r \big|^2 \big) \\
		&= \Erw \bigg( \Erw \bigg( 
		\sum_{s,t=1}^M \big( S(X^{(\Se)}) \big)_{r,s} \Psi^{2,(\Se)}_s
		\big( S(X^{(\Se)}) \big)_{r,t} \Psi^{2,(\Se)}_t \biggMid X^{(\Se)} \bigg) \bigg) \\
		&= \Erw \bigg( \sum_{s=1}^M \big( S(X^{(\Se)}) \big)_{r,s}^2 \bigg) \\
		&= \Erw \big( \big\| C(r) \big( ({\Sigma^{2,(\Se)}(X^{(\Se)})})^{1/2} 
		- ({\Sigma^{2,\infty}})^{1/2} \big) \big\|_F^2 \big) \, .
	\end{align*}
	Together with
	Proposition~\ref{Prop-dist-SigmaN-SigmaInfty}~(\ref{Prop-dist-SigmaN-SigmaInfty-i})
	and Lemma~\ref{Lem-Sqrt-A-B-estimate} we obtain the estimate
	\begin{align*}
		&\Erw \big( \big\| C(r) \big( ({\Sigma^{2,(\Se)}(X^{(\Se)})})^{1/2} 
		- ({\Sigma^{2,\infty}})^{1/2} \big) \big\|_F^2 \big) \\
		&\leq \bigg( \frac{h^2}{2 \pi^2} \sum_{k=\Se+1}^{\infty} \frac{1}{k^2} \bigg)^{-1}
		\Erw \big( \big\| C(r) \big( {\Sigma^{2,(\Se)}(X^{(\Se)})}
		- {\Sigma^{2,\infty}} \big) \big\|_F^2 \big) \\
		&= \bigg(\frac{h^2}{2 \pi^2} \sum_{k=\Se+1}^{\infty} \frac{1}{k^2} \bigg)^{-1}
		\sum_{s=1}^M \Erw \Big( \big| \big( {\Sigma^{2,(\Se)}(X^{(\Se)})}
		- {\Sigma^{2,\infty}} \big)_{r,s} \big|^2 \Big) \\
		&\leq \frac{h^4 m}{8 \pi^4} \bigg(\frac{h^2}{2 \pi^2} \sum_{k=\Se+1}^{\infty}
		\frac{1}{k^2} \bigg)^{-1} \sum_{k=\Se+1}^{\infty} k^{-4} \\
		&= \frac{h^2 m}{4 \pi^2} \bigg( \sum_{k=\Se+1}^{\infty}
		\frac{1}{k^2} \bigg)^{-1} \sum_{k=\Se+1}^{\infty} k^{-4}
	\end{align*}
	which proves (i). Due to $I_{(i,i)}(h) = I_{(i,i)}^{(\Se)}(h)$ for $i=1, \ldots, m$, 
	(ii) follows directly from (i).
\end{proof}
\begin{cor} \label{Cor-I-In-Fehler-Estimate}
	Let $\Se \in \mathbb{N}$ and let ${I}_{(i,j)}^{(\Se)}(h)$
	be the approximation defined in \eqref{Approximation-vec-Version-hat-Iij} 
	of the iterated stochastic integral $I_{(i,j)}(h)$ for $1 \leq i, j \leq m$. 
	Then, it holds:
	\begin{enumerate}[(i)]
		\item 
		\begin{equation} \label{Cor-I-In-Fehler-Estimate-eqn1}
			\max_{1 \leq i,j \leq m} \big( \Erw \big( | I_{(i,j)}(h) - I_{(i,j)}^{(\Se)}(h) |^2 
			\big) \big)^{1/2}
			\leq \frac{\sqrt{m} \, h}{\sqrt{12} \, \pi \, \Se} 
		\end{equation}
		\item 
		\begin{equation} \label{Cor-I-In-Fehler-Estimate-eqn2}
            \big( \Erw \big( \| I(h) - I^{(\Se)}(h) \|_F^2 \big) \big)^{1/2}
            \leq \frac{\sqrt{m-1} \, m \, h}{\sqrt{12} \, \pi \, \Se} 
        \end{equation}
	\end{enumerate}
\end{cor}
\begin{proof}
	The assertions follow from Theorem~\ref{Theo-I-In-Fehler-Estimate} and the
	estimate
	\begin{align*}
		\frac{ \sum_{k=\Se+1}^{\infty} \frac{1}{k^4} }{ \sum_{k=\Se+1}^{\infty} 
		\frac{1}{k^2} } &\leq \frac{ \int_{\Se + \frac{1}{2}}^{\infty} \frac{1}{x^4} \,
		\mathrm{d}x }{ \int_{\Se + \frac{3}{4}}^{\infty} \frac{1}{x^2} \, \mathrm{d}x }
		= \frac{\Se + \frac{3}{4}}{3 (\Se + \frac{1}{2})^3} 
		\leq \frac{1}{3 \Se^2}
	\end{align*}
	for all $\Se \in \mathbb{N}$, see also \cite{MR1843055} for this estimate.
\end{proof}
\subsection{Error estimates in $L^p$-norm}
\noindent
Next to the error estimates in $L^2(\Omega)$-norm, we also 
give some error estimate in $L^p(\Omega)$-norm for arbitrary $p>2$ in the
following. Therefore, we first prove $L^p$ convergence for the
truncated Fourier series approach proposed in \cite{MR1178485,MR1214374,Mil95},
which is a new result on its own.
\begin{prop} \label{Prop-Iij-FS-convergence-Lp}
	Let $\Se \in \mathbb{N}$ and $p \geq 2$. Then, for the truncated Fourier series
	approximation $I_{(i,j)}^{FS,(\Se)}(h)$ in \eqref{Iij-Truncated-Fourier-Series-Alg} 
	it holds
	\begin{align} \label{Prop-Iij-FS-convergence-Lp-eqn}
		\max_{1 \leq i,j \leq m} \big( \Erw \big( \big| I_{(i,j)}(h) 
		- I_{(i,j)}^{FS,(\Se)}(h) \big|^p \big) \big)^{1/p} 
		\leq \frac{(p-1) \, h}{\sqrt{2} \pi} \Big( \Gamma \Big( \frac{p}{2}+1 \Big)
		\Big)^{1/p} \bigg( \frac{\pi^2}{6} - \sum_{k=1}^{\Se} \frac{1}{k^2} \bigg)^{1/2}
	\end{align}
	with equality for $p=2$ if $i \neq j$ and $\big( \Erw \big( \big| I_{(i,i)}(h) 
	- I_{(i,i)}^{FS,(\Se)}(h) \big|^p \big) \big)^{1/p} =0$ for $i \in \{1, \ldots, m\}$.
\end{prop}
\noindent
For the proof of Proposition~\ref{Prop-Iij-FS-convergence-Lp}, the following
lemma on the $p$th moment of $\chi^2_2$ distributed random variables is needed
first.
\begin{lem} \label{Lem-Momente-Chi2}
	Let $X$ and $Y$ be independent $\Nd(0,1)$ distributed random variables
	on $(\Omega, \mathcal{F}, \Prob)$, let $p > -1$ and $c \in \mathbb{R}$. 
	Then, for the $\chi^2$ distributed 
	random variable $X^2+Y^2$ with $2$ degrees of freedom it holds
	\begin{align*}
	\Erw \big( \big| X^2+Y^2 - c \big|^p \big) 
	= 2^p \, e^{-\frac{c}{2}} \bigg( \Gamma(p+1) + \int_0^{\frac{c}{2}} 
	|t|^p \, e^t \, \mathrm{d}t \bigg) \, .
	\end{align*}
\end{lem}
\begin{proof}
	Let $p > -1$ and $c \in \mathbb{R}$ be fixed. Using polar coordinates, we
	calculate due to the independence of $X$ and $Y$ that
	\begin{align*}
	\Erw \big( \big| X^2+Y^2 - c \big|^p \big) &= \frac{1}{2 \pi} 
	\int_{-\infty}^{\infty} \int_{-\infty}^{\infty} \big| x^2 + y^2 -c \big|^p 
	\, e^{-\frac{1}{2} ( x^2+y^2 )} \, \mathrm{dx} \, \mathrm{d}y \\
	&= \frac{2}{\pi} \int_{0}^{\infty} \int_{0}^{\infty} \big| x^2 + y^2 -c \big|^p 
	\, e^{-\frac{1}{2} ( x^2+y^2 )} \, \mathrm{dx} \, \mathrm{d}y \\
	&= \frac{2}{\pi} \int_{0}^{\frac{\pi}{2}} \int_{0}^{\infty} \big| r^2 -c \big|^p 
	\, r \, e^{-\frac{1}{2} r^2} \, \mathrm{dr} \, \mathrm{d}\varphi \\
	&= \int_{0}^{\infty} \big| r^2 -c \big|^p 
	\, r \, e^{-\frac{1}{2} r^2} \, \mathrm{dr} \, .
	\end{align*}
	Substituting $z = \frac{1}{2} (r^2 -c)$, we get
	\begin{align*}
	\int_{0}^{\infty} \big| r^2 -c \big|^p \, r \, e^{-\frac{1}{2} r^2} \, \mathrm{dr}
	&= 2^p \, e^{-\frac{c}{2}} \bigg( \int_{-\frac{c}{2}}^0 |z|^p \, e^{-z} \,
	\mathrm{d}z 
	+ \int_0^{\infty} z^p \, e^{-z} \, \mathrm{d}z \bigg)
	\, .
	\end{align*} 
	Substituting $t=-z$ in the first integral and expressing the second integral
	by the gamma function completes the proof.
\end{proof}
\noindent
Note that one may choose $c=0$ (or $c=\Erw \big(X^2+Y^2 \big) = 2$) in 
order to get a formula for the $p$th absolute moment of a (centered) 
$\chi^2_2$ random variable. 
\begin{proof}[Proof of Proposition~\ref{Prop-Iij-FS-convergence-Lp}]
	Let $\Se \in \mathbb{N}$ and $p \geq 2$ be arbitrarily fixed. Then, it holds for
	$i,j \in \{1, \ldots, m\}$ with $i \neq j$ that
	\begin{align*}
		I_{(i,j)}(h) - I_{(i,j)}^{FS,(\Se)}(h) 
		&= \int_0^h \int_0^s \, \mathrm{d}W_u^i \, \mathrm{d}W_s^j \\
		&\quad - \Big( \frac{1}{2} W_h^i \, W_h^j
		+ \frac{a_{i,0}}{2} W_h^j - \frac{a_{j,0}}{2} W_h^i + \pi \sum_{k=1}^{\Se}
		k \big( a_{i,k} \, b_{j,k} - a_{j,k} \, b_{i,k} \big) \Big) \\
		&= \int_0^h W_s^i - \Big( 
		\frac{s}{h} W_h^i + \frac{a_{i,0}}{2} + \sum_{k=1}^{\Se} a_{i,k} \, \cos 
		\Big( \frac{2 \pi k s}{h} \Big) + b_{i,k} \, \sin \Big( \frac{2 \pi k s}{h} \Big) \Big)
		\, \mathrm{d}W_s^j  .
	\end{align*}
	With the Burkholder-Davis-Gundy inequality we get
	\begin{align*}
		&\Erw \big( \big| I_{(i,j)}(h) - I_{(i,j)}^{FS,(\Se)}(h) \big|^p \big) \\
		&\leq (p-1)^p \, \Erw \bigg( \bigg( \int_0^h \bigg| W_s^i - \frac{s}{h} W_h^i
		- \Big( \frac{a_{i,0}}{2} + \sum_{k=1}^{\Se} a_{i,k} \, \cos \Big( \frac{2 \pi k s}{h}
		\Big) + b_{i,k} \, \sin \Big( \frac{2 \pi k s}{h} \Big) \Big) \bigg|^2 \, \mathrm{ds} 
		\bigg)^{p/2} \bigg) .
	\end{align*}
	Taking into account that $\big\{ \frac{1}{\sqrt{h}}, \sqrt{\frac{2}{h}} \sin \big( 
	\frac{2 \pi k s}{h} \big), \sqrt{\frac{2}{h}} \cos \big( \frac{2 \pi k s}{h} \big) , 
	\ k \in \mathbb{N}, \ s \in [0,h] \big\}$ is an orthonormal basis of $L^2([0,h])$,
	that $\tilde{W}^i = (W_s^i - \frac{s}{h} W_h^i)_{s \in [0,h]}$ has $\Prob$-a.s.\
	continuous
	paths and thus belongs $\Prob$-a.s.\ to $L^2([0,h])$, we get with Parseval's
	equality, Lebesgue's dominated convergence theorem and triangle inequality
	for the right hand side
	\begin{align*}
		&(p-1)^p \, \Erw \bigg( \bigg( \int_0^h \bigg| W_s^i - \frac{s}{h} W_h^i
		- \Big( \frac{a_{i,0}}{2} + \sum_{k=1}^{\Se} a_{i,k} \, \cos \Big( \frac{2 \pi k s}{h}
		\Big) + b_{i,k} \, \sin \Big( \frac{2 \pi k s}{h} \Big) \Big) \bigg|^2 \, \mathrm{ds} 
		\bigg)^{p/2} \bigg) \\
		&= (p-1)^p \, \Big( \frac{h}{2} \Big)^{p/2} \Erw \bigg( \bigg( 
		\sum_{k=\Se +1}^{\infty} |a_{i,k} |^2 + |b_{i,k} |^2 \bigg)^{p/2} \bigg) \\
		&\leq (p-1)^p \, \Big( \frac{h}{2} \Big)^{p/2} \bigg( \sum_{k=\Se +1}^{\infty} 
		\Big( \Erw \Big( \big( |a_{i,k} |^2 + |b_{i,k} |^2 \big)^{p/2} \Big) \Big)^{2/p} 
		\bigg)^{p/2}
	\end{align*}
	because $\Erw ( \| \tilde{W}^i \|_{L^2([0,h])}^{p} ) < \infty$.
	Since $a_{i,k}, b_{i,k} \sim \Nd \big( 0, \frac{h}{2 \pi^2 k^2} \big)$ are i.i.d., 
	it follows that $X_{i,k} = \frac{\sqrt{2} \pi k}{\sqrt{h}} a_{i,k} \sim \Nd(0,1)$ and 
	$Y_{i,k} = \frac{\sqrt{2} \pi k}{\sqrt{h}} b_{i,k} \sim \Nd(0,1)$ are i.i.d. Then, 
	due to Lemma~\ref{Lem-Momente-Chi2} it holds
	\begin{align*}
		&(p-1)^p \, \Big( \frac{h}{2} \Big)^{p/2} \bigg( \sum_{k=\Se+1}^{\infty} 
		\Big( \Erw \Big( \big( |a_{i,k} |^2 + |b_{i,k} |^2 \big)^{p/2} \Big) \Big)^{2/p} 
		\bigg)^{p/2} \\
		&= (p-1)^p \, \Big( \frac{h}{2} \Big)^{p/2} \bigg( \sum_{k=\Se +1}^{\infty} 
		\frac{h}{2 \pi^2 k^2} \Big( \Erw \Big( \big( X_{i,k}^2 + Y_{i,k}^2 \big)^{p/2} 
		\Big) \Big)^{2/p} \bigg)^{p/2} \\
		&= (p-1)^p \, \Big( \frac{h}{2 \pi} \Big)^{p} \bigg( \sum_{k=\Se +1}^{\infty} 
		\frac{1}{k^2} \bigg)^{p/2} 2^{p/2} \, \Gamma \Big( \frac{p}{2} +1 \Big) \, .
	\end{align*}
	The case $i=j$ follows directly due to $A_{(i,i)}^{(\Se)}(h)=0$ and
	$R_{(i,i)}^{1,(\Se)}(h) = 0$ for all $\Se \in \mathbb{N}$.
\end{proof}
\noindent
Now, we state the main result of this section on the approximation error 
of $\hat{I}^{(\Se)}(h)$ in \eqref{Approximation-vec-Version-hat-Iij} for the 
iterated stochastic integral $\hat{I}(h)$ in the $L^p(\Omega)$-norm.
\begin{thm} \label{Theo-Lp-Error-Estimates}
    Let $p > 2$ and $\Se \in \mathbb{N}$. Then, for the approximation
    ${I}_{(i,j)}^{(\Se)}(h)$ defined in \eqref{Approximation-vec-Version-hat-Iij}
    of the iterated stochastic integral $I_{(i,j)}(h)$ for $1 \leq i, j \leq m$ 
    and for the corresponding $m \times m$ matrices 
    $I(h) = (I_{(i,j)}(h))_{1 \leq i,j \leq m}$ and 
    $I^{(\Se)}(h) = ( I_{(i,j)}^{(\Se)}(h) )_{1 \leq i,j \leq m}$ it holds
    \begin{enumerate}[(i)]
        \item \label{Theo-Lp-Error-Estimates-i}
        \begin{equation} 
        	\max_{1 \leq i,j \leq m} \big( \Erw \big( | I_{(i,j)}(h) - I_{(i,j)}^{(\Se)}(h) |^p 
        	\big) \big)^{1/p}
        	\leq c_{m,p} \frac{\sqrt{p-1} h}{\pi^{\frac{2p+1}{2p}}} \, 
        	\bigg( \frac{\sum_{k=\Se+1}^{\infty} k^{-4}}{\sum_{k=\Se+1}^{\infty} k^{-2} }
        	\bigg)^{1/2} , 
        \end{equation}
        \item \label{Theo-Lp-Error-Estimates-ii}
        \begin{equation} 
        	\big( \Erw \big( \| I(h) - I^{(\Se)}(h) \|_F^p \big) \big)^{1/p}
        	\leq c_{m,p} \frac{\sqrt{(p-1) (m^2-m)} h}{\pi^{\frac{2p+1}{2p}}} \, 
        	\bigg( \frac{\sum_{k=\Se+1}^{\infty} k^{-4}}{\sum_{k=\Se+1}^{\infty} k^{-2} }
        	\bigg)^{1/2} ,
        \end{equation}
    \end{enumerate}
	where the constant $c_{m,p}$ is given by
	\begin{align*}
		c_{m,p} &= \big( \Gamma \big( \tfrac{p+1}{2} \big) \big)^{1/p}
		\bigg( e^{-2/p} \bigg( \Gamma(p+1) + \frac{e}{p+1} \bigg)^{2/p}
		+ \frac{2m-4}{\pi^{2/p}} \, \big( 
		\Gamma \big( \tfrac{p+1}{2} \big) \big)^{4/p} \bigg)^{1/2} .
	\end{align*}
\end{thm}
\noindent
For the proof of Theorem~\ref{Theo-Lp-Error-Estimates} we need the following
result that can be easily proved, see also \cite[p.~5 (1.1)]{MR1474726} and \cite[Lemma~V.20]{PleisDiss20}.
\begin{lem} \label{Lem-Momente-Normalvert}
	Let $Z \sim \Nd(0, \sigma^2)$ be a real-valued Gaussian random variable with
	some finite $\sigma>0$. Then, for $p \in {[1,\infty[}$ it holds
	\begin{equation*}
		\Erw( |Z|^p ) = \frac{ (\sqrt{2} \sigma )^p }{ \sqrt{\pi} } \, 
		\Gamma \bigg( \frac{p+1}{2} \bigg) \, .
	\end{equation*}
\end{lem}
\begin{proof}
	Let $p \geq 1$ and $\sigma>0$ be arbitrarily fixed. 
	Substituting $x = \frac{z^2}{2 \sigma^2}$, it follows that
	\begin{align*}
		\Erw (|Z|^p) &= \frac{1}{\sqrt{2 \pi} \sigma} \int_{-\infty}^{\infty} |z|^p 
		e^{-\frac{z^2}{2 \sigma^2}} \, \mathrm{d}z
		= \frac{2}{\sqrt{2 \pi} \sigma} \int_{0}^{\infty} z^p 
		e^{-\frac{z^2}{2 \sigma^2}} \, \mathrm{d}z 
		= \frac{(2 \sigma^2)^{\frac{p+1}{2}}}{\sqrt{2 \pi} \sigma} 
		\int_0^{\infty} x^{\frac{p-1}{2}} \, e^{-x} \, \mathrm{d}x \, .
	\end{align*}
	For $p=1$ the assertion follows directly because $\Gamma(1)=1$.
	In case of $p>1$, the assertion follows due to $\Gamma (s)
	= \int_0^{\infty} t^{s-1} \, e^{-t} \, \mathrm{d}t$ for $s >0$.
\end{proof}
\noindent
Now, we are prepared for the proof of the error estimate in $L^p(\Omega)$-norm
stated in Theorem~\ref{Theo-Lp-Error-Estimates}.
\begin{proof}[Proof of Theorem~\ref{Theo-Lp-Error-Estimates}]
    Let $p >2$ and $\Se \in \mathbb{N}$ arbitrarily fixed. Then, we get due to
    Proposition~\ref{Prop-Iij-FS-convergence-Lp} that
    \begin{align*}
    	\max_{1 \leq i,j \leq m} \Erw \big( | I_{(i,j)}(h) - I_{(i,j)}^{(\Se)}(h) |^p \big)
    	&= \max_{1 \leq r \leq M} \Erw \big( \big| \big( 
    	\hat{A}(h) - \big(\hat{A}^{(\Se)}(h)
    	+ \hat{A}^{1,(\Se)}(h) + \hat{A}^{2,(\Se)}(h) \big) \big)_r \big|^p \big) \\ 
    	&= \max_{1 \leq r \leq M} \Erw \big( \big| \big( 
    	\hat{R}^{2,(\Se)}(h) - \hat{A}^{2,(\Se)}(h) \big)_r \big|^p \big) \\
    	&= \max_{1 \leq r \leq M} \Erw \big( \big| \big(
    	({\Sigma^{2,(\Se)}(X^{(\Se)})})^{1/2} \Psi^{2,(\Se)}
    	- ({\Sigma^{2,\infty}})^{1/2} \Psi^{2,(n)} 
    	\big)_r \big|^p \big) 
		\, .
    \end{align*}
    Let $r \in \{1, \ldots, M\}$ arbitrarily fixed. Given $X^{(\Se)}$, the real-valued
    random variable 
    \begin{align*}
    	\varphi_r \big( X^{(\Se)}, \Psi^{2,(\Se)} \big) 
    	&= \big( ({\Sigma^{2,(\Se)}(X^{(\Se)})})^{1/2} \Psi^{2,(\Se)}
    	- ({\Sigma^{2,\infty}})^{1/2} \Psi^{2,(n)} \big)_r \\
    	&= \sum_{s=1}^M \big(
    	({\Sigma^{2,(\Se)}(X^{(n)})})^{1/2} - ({\Sigma^{2,\infty}})^{1/2} \big)_{r,s} 
    	\Psi^{2,(\Se)}_s
    \end{align*}
    is conditionally Gaussian distributed with conditional expectation $0$ and
    conditional variance 
    \begin{align*}
    	\sigma_{\varphi_r}^2 (X^{(\Se)}) = \sum_{s=1}^M \big(
    	({\Sigma^{2,(\Se)}(X^{(\Se)})})^{1/2} - ({\Sigma^{2,\infty}})^{1/2} \big)_{r,s}^2 \, .
    \end{align*}
    First of all, we prove that $\big| \varphi_r \big( X^{(\Se)}, \Psi^{2,(\Se)} \big)
    \big|^p \in L^1(\Omega)$.
    Let $C(r) \in \mathbb{R}^{1 \times M}$ with $C_{1,r}(r)=1$ and $C_{1,l}(r)=0$ 
    for $l \neq r$.
	Then, making use of Lemma~\ref{Lem-Sqrt-A-B-estimate} and taking into
	account that the minimal eigenvalue of $\Sigma^{2,\infty}$ is 
	$\lambda_{\min}^{2,\infty} = \frac{h^2}{2 \pi^2} \sum_{k=\Se+1}^{\infty}
	\frac{1}{k^2}$, we obtain
	\begin{align*}
	    \Erw \big( \big| \varphi_r \big( X^{(\Se)}, \Psi^{2,(\Se)} \big) \big|^p \big) 
		&= \Erw \big( \big\| C(r) \, \big( ({\Sigma^{2,(\Se)}(X^{(\Se)})})^{1/2} 
		- ({\Sigma^{2,\infty}})^{1/2} \big) \, \Psi^{2,(\Se)} \big\|_2^{p} \big) \\
		&\leq \Erw \big( \big\| C(r) \, \big( ({\Sigma^{2,(\Se)}(X^{(\Se)})})^{1/2} 
		- ({\Sigma^{2,\infty}})^{1/2} \big) \big\|_F^p \big) \, 
		\Erw \big( \big\| \Psi^{2,(\Se)} \big\|_2^{p} \big) \\
		&\leq \big( \lambda_{\min}^{2,\infty} \big)^{-p/2} \,
		\Erw \big( \big\| C(r) \, \big( \Sigma^{2,(\Se)}(X^{(\Se)})
		- \Sigma^{2,\infty} \big) \big\|_F^{p} \big) \, 
		\Erw \big( \big\| \Psi^{2,(\Se)} \big\|_2^{p} \big) \\
		&= \big( \lambda_{\min}^{2,\infty} \big)^{-p/2} \,
		\Erw \bigg( \bigg| \sum_{s=1}^M \big( \Sigma^{2,(\Se)}(X^{(\Se)})
		- \Sigma^{2,\infty} \big)_{r,s}^2 \bigg|^{p/2} \bigg) \, 
		\Erw \big( \big\| \Psi^{2,(\Se)} \big\|_2^{p} \big)
	\end{align*}
	with $\Erw ( \| \Psi^{2,(\Se)} \|_2^{p} ) 
	\leq M^{p/2-1} \sum_{s=1}^M \Erw ( | \Psi^{2,(\Se)}_s |^p ) 
	< \infty$ due to Lemma~\ref{Lem-Momente-Normalvert}. 
	Moreover, with the triangle inequality it holds 
	\begin{align*}
		\Erw \bigg( \bigg| \sum_{s=1}^M \big( \Sigma^{2,(\Se)}(X^{(\Se)})
		- \Sigma^{2,\infty} \big)_{r,s}^2 \bigg|^{p/2} \bigg)
		&= \bigg\| \sum_{s=1}^M \big( \Sigma^{2,(\Se)}(X^{(\Se)})
		- \Sigma^{2,\infty} \big)_{r,s}^2 \bigg\|_{L^{p/2}(\Omega)}^{p/2} \\
		&\leq \bigg( \sum_{s=1}^M \big\| \big( \Sigma^{2,(\Se)}(X^{(\Se)})
		- \Sigma^{2,\infty} \big)_{r,s}^2 \big\|_{L^{p/2}(\Omega)} \bigg)^{p/2} \\
		&= \bigg( \sum_{s=1}^M \big\| \big( \Sigma^{2,(\Se)}(X^{(\Se)})
		- \Sigma^{2,\infty} \big)_{r,s} \big\|_{L^{p}(\Omega)}^2 \bigg)^{p/2} \, .
	\end{align*}
    Making use of the structure of the matrix $\Sigma^{2,(\Se)}(X^{(\Se)})$
    detailed in Appendix~\ref{Appendix-A} and applying
    Lemma~\ref{Lem-Momente-Normalvert}, we get
    \begin{equation} \label{Proof-Lp-Estimate-ZN-upb}
    	\begin{split}
    	\big\| \big( \Sigma^{2,(\Se)}(X^{(\Se)}) - \Sigma^{2,\infty} \big)_{r,s}
    	\big\|_{L^p(\Omega)}
    	&\leq \frac{h^2}{4 \pi^2} \sum_{k=\Se+1}^{\infty} \frac{1}{k^2}
    	\, \big\| \big(H_m \Sigma(X_1) H_m^{\Tt} - 2 \, I_M \big)_{r,s} 
    	\big\|_{L^p(\Omega)} \\
    	&\leq \frac{h^2}{4 \pi^2} \, \frac{\pi^2}{6} \max \big( 
    	\big\| X_{1,1}^2 + X_{2,1}^2 -2 \big\|_{L^p(\Omega)}, 
    	\big\| X_{1,1} \, X_{2,1}	\big\|_{L^p(\Omega)} \big) \\
    	&\leq \frac{h^2}{24} \max \bigg( 2  \bigg( \frac{\sqrt{2}^p}{\sqrt{\pi}} 
    	\Gamma \big( \tfrac{p+1}{2} \big) \bigg)^{1/p} + 2, \bigg( \frac{\sqrt{2}^p}{\sqrt{\pi}} 
    	\Gamma \big( \tfrac{p+1}{2} \big) \bigg)^{2/p} \bigg) .
    	\end{split}
    \end{equation}
    Thus, $\Erw (| \varphi_r \big( X^{(\Se)}, \Psi^{2,(\Se)} \big) |^p )$ is uniformly
    bounded for all $1 \leq r \leq M$ and therefore it holds
    $| \varphi_r \big( X^{(\Se)}, \Psi^{2,(\Se)} \big) |^p \in L^1(\Omega)$ for all
    $1 \leq r \leq M$.
	\\ \\
    Since $X^{(\Se)}$ and $\Psi^{2,(\Se)}$ are stochastically independent and
    because $|\varphi_r \big( X^{(\Se)}, \Psi^{2,(\Se)} \big) \big|^p \in
    L^1(\Omega)$, it follows by the substitution property of the conditional
    expectation \cite[Theorem~2.10]{MR3618289} that 
    \begin{align*}
    	\Erw \big( \big|
    	\varphi_r \big( X^{(\Se)}, \Psi^{2,(\Se)} \big) \big|^p \bigMid X^{(\Se)} \big) 
    	= h \big( X^{(\Se)} \big)
    \end{align*}
    with $h(x) = \Erw \big( \big| \varphi_r \big( x, \Psi^{2,(\Se)} \big) \big|^p \big)$.
    Since $\varphi_r \big( x, \Psi^{2,(\Se)} \big)$
    has a Gaussian distribution with expectation $0$ and variance
    $\sigma_{\varphi_r}^2(x)$ it follows with Lemma~\ref{Lem-Momente-Normalvert}
    that $h(x) = \frac{( \sqrt{2} \sigma_{\varphi_r}(x) )^p}{\sqrt{\pi}} 
    \Gamma \big( \frac{p+1}{2} \big)$. Thus, it holds
	\begin{align*}
		&\Erw \big( \big| \big(
		({\Sigma^{2,(\Se)}(X^{(\Se)})})^{1/2} \Psi^{2,(\Se)}
		- ({\Sigma^{2,\infty}})^{1/2} \Psi^{2,(\Se)} \big)_r \big|^p \big) \\
		&= \Erw \bigg( \Erw \bigg(
		\bigg| \sum_{s=1}^M \big( ({\Sigma^{2,(\Se)}(X^{(\Se)})})^{1/2}
		- ({\Sigma^{2,\infty}})^{1/2} \big)_{r,s}  \Psi^{2,(\Se)}_s \bigg|^p 
		\biggMid X^{(\Se)} \bigg) \bigg) \\
		&= \frac{ 2^{p/2} }{ \sqrt{\pi} } \, \Gamma \bigg( \frac{p+1}{2} \bigg) \,
		\Erw \bigg( \bigg| \sum_{s=1}^M \big( ({\Sigma^{2,(\Se)}(X^{(\Se)})})^{1/2}
		- ({\Sigma^{2,\infty}})^{1/2} \big)_{r,s}^2 \bigg|^{p/2} \bigg) \\
		&= \frac{ 2^{p/2} }{ \sqrt{\pi} } \, \Gamma \bigg( \frac{p+1}{2} \bigg) \,
		\Erw \big( \big\| C(r) \, \big( ({\Sigma^{2,(\Se)}(X^{(\Se)})})^{1/2}
		- ({\Sigma^{2,\infty}})^{1/2} \big) \big\|_F^{p} \big) \, .
	\end{align*}
	Next, we make use of Lemma~\ref{Lem-Sqrt-A-B-estimate} taking into account 
	that the minimal eigenvalue of $\Sigma^{2,\infty}$ is 
	$\lambda_{\min}^{2,\infty} = \frac{h^2}{2 \pi^2} \sum_{k=\Se+1}^{\infty}
	\frac{1}{k^2}$. Then, we obtain
	\begin{align*}
		&\frac{ 2^{p/2} }{ \sqrt{\pi} } \, \Gamma \bigg( \frac{p+1}{2} \bigg) \,
		\Erw \big( \big\| C(r) \, \big( ({\Sigma^{2,(\Se)}(X^{(\Se)})})^{1/2}
		- ({\Sigma^{2,\infty}})^{1/2} \big) \big\|_F^{p} \big) \\
		&\leq \frac{ 2^{p/2} }{ \sqrt{\pi} } \, \Gamma \bigg( \frac{p+1}{2} \bigg) \,
		\big( \lambda_{\min}^{2,\infty} \big)^{-p/2} \,
		\Erw \big( \big\| C(r) \, \big( \Sigma^{2,(\Se)}(X^{(\Se)})
		- \Sigma^{2,\infty} \big) \big\|_F^{p} \big) \\
		&= \frac{ 2^{p/2} }{ \sqrt{\pi} } \, \Gamma \bigg( \frac{p+1}{2} \bigg) \,
		\big( \lambda_{\min}^{2,\infty} \big)^{-p/2} \,
		\Erw \bigg( \bigg| \sum_{s=1}^M \big( \Sigma^{2,(\Se)}(X^{(\Se)})
		- \Sigma^{2,\infty} \big)_{r,s}^2 \bigg|^{p/2} \bigg) \\
		&\leq \frac{ 2^{p/2} }{ \sqrt{\pi} } \, \Gamma \bigg( \frac{p+1}{2} \bigg) \,
		\big( \lambda_{\min}^{2,\infty} \big)^{-p/2} \,
		\bigg( \sum_{s=1}^M \bigg( \Erw \bigg( \bigg| \big( \Sigma^{2,(\Se)}(X^{(\Se)})
		- \Sigma^{2,\infty} \big)_{r,s}^2 \bigg|^{p/2} \bigg) \bigg)^{2/p} \bigg)^{p/2} \\
		&= \frac{ 2^{p/2} }{ \sqrt{\pi} } \, \Gamma \bigg( \frac{p+1}{2} \bigg) \,
		\big( \lambda_{\min}^{2,\infty} \big)^{-p/2} \,
		\bigg( \frac{h^2}{4 \pi^2} \bigg)^p \, \bigg( \sum_{s=1}^M
		\bigg\| \sum_{k=\Se+1}^{\infty}
		\frac{1}{k^2} \big(H_m \Sigma(X_k) H_m^{\Tt} - 2 \, I_M \big)_{r,s}
		\bigg\|_{L^p(\Omega)}^{2} \bigg)^{p/2}
	\end{align*}
	by the triangle inequality. 
	Due to estimate \eqref{Proof-Lp-Estimate-ZN-upb} it follows that
	$(Z_N^{(r,s)})_{N \geq \Se+1}$ with $Z_N^{(r,s)} = \sum_{k=\Se+1}^N
	\frac{1}{k^2} \big(H_m \Sigma(X_k) H_m^{\Tt} - 2 \, I_M \big)_{r,s}$ is a
	martingale 
	that is uniformly bounded in $L^p(\Omega)$ for all $N \geq \Se+1$. 
	Thus, a Burkholder type inequality as in \cite[Prop.~2.5]{MR1331198} or
	\cite[Prop.~2.1]{PlRoe20arXiv2006.16048v1} can be applied to obtain
	\begin{align*}
		&\bigg( \sum_{s=1}^M \bigg\| \sum_{k=\Se+1}^{\infty}
		\frac{1}{k^2} \big(H_m \Sigma(X_k) H_m^{\Tt} - 2 \, I_M \big)_{r,s}
		\bigg\|_{L^p(\Omega)}^{2} \bigg)^{p/2} \\
		&= \bigg( \sum_{s=1}^M \lim_{N \to \infty}
		\bigg\| \sum_{k=\Se+1}^{N}
		\frac{1}{k^2} \big(H_m \Sigma(X_k) H_m^{\Tt} - 2 \, I_M \big)_{r,s}
		\bigg\|_{L^p(\Omega)}^{2} \bigg)^{p/2} \\
		&\leq \bigg( \sum_{s=1}^M \lim_{N \to \infty} (p-1) \sum_{k=\Se+1}^N 
		\bigg\| \frac{1}{k^2} \big(H_m \Sigma(X_k) H_m^{\Tt} - 2 \, I_M \big)_{r,s} 
		\bigg\|_{L^p(\Omega)}^2 \bigg)^{p/2} \\
		&= \bigg( (p-1) \sum_{k=\Se+1}^{\infty}
		\frac{1}{k^4} \bigg)^{p/2} \, \bigg( \sum_{s=1}^M 
		\big\| \big(H_m \Sigma(X_1) H_m^{\Tt} - 2 \, I_M \big)_{r,s} 
		\big\|_{L^p(\Omega)}^2 \bigg)^{p/2} \\
		&= \bigg( (p-1) \sum_{k=\Se+1}^{\infty}
		\frac{1}{k^4} \bigg)^{p/2} \, \Big( \| X_{1,1}^2 + X_{2,1}^2 - 2 \|_{L^p(\Omega)}^2
		+ (2m-4) \| X_{1,1} X_{2,1} \|_{L^p(\Omega)}^2 \Big)^{p/2} \, .
	\end{align*}
	Now, Lemma~\ref{Lem-Momente-Chi2} and
	Lemma~\ref{Lem-Momente-Normalvert} are applied. Further, due to 
	$e^t \leq e^1$ for all $t \in [0,1]$, it follows that
	\begin{align*}
		&\Big( \| X_{1,1}^2 + X_{2,1}^2 - 2 \|_{L^p(\Omega)}^2 + (2m-4) 
		\| X_{1,1} X_{2,1} \|_{L^p(\Omega)}^2 \Big)^{p/2} \\
		&= \bigg( \frac{4}{e^{2/p}} \bigg( \Gamma(p+1) + \int_0^1 t^p \, e^t \,
		\mathrm{d}t \bigg)^{2/p}
		+ (2m-4) \bigg( \frac{ (\sqrt{2})^p }{ \sqrt{\pi} } \, 
		\Gamma \bigg( \frac{p+1}{2} \bigg) \bigg)^{4/p} \bigg)^{p/2} \\
		&\leq \bigg( \frac{4}{e^{2/p}} \bigg( \Gamma(p+1) + \frac{e}{p+1} \bigg)^{2/p}
		+ (2m-4) \, \frac{4}{\pi^{2/p}} \, \bigg( 
		\Gamma \bigg( \frac{p+1}{2} \bigg) \bigg)^{4/p} \bigg)^{p/2} \, .
	\end{align*}
	Thus, (\ref{Theo-Lp-Error-Estimates-i}) is proved. For (\ref{Theo-Lp-Error-Estimates-ii}), observe that 
	\begin{align*}
		\big( \Erw \big( \| I(h) - I^{(\Se)}(h) \|_F^p \big) \big)^{1/p} 
		&= \bigg\| \sum_{i,j=1}^m \big| I_{(i,j)}(h) - I_{(i,j)}^{(\Se)}(h) \big|^2 
		\bigg\|_{L^{p/2}(\Omega)}^{1/2} \\
		&\leq \bigg( \sum_{i,j=1}^m \bigg\| \big| I_{(i,j)}(h) - I_{(i,j)}^{(\Se)}(h) \big|^2 
		\bigg\|_{L^{p/2}(\Omega)} \bigg)^{1/2} \\
		&\leq (m^2-m)^{1/2} \, \max_{1 \leq i,j \leq m} \big\| I_{(i,j)}(h) - I_{(i,j)}^{(\Se)}(h)
		\big\|_{L^p(\Omega)}
	\end{align*}
	because $I_{(i,i)}(h) - I_{(i,i)}^{(\Se)}(h) = 0$ $\Prob$-a.s.\ for all $i=1, \ldots, m$.
\end{proof}
\noindent
From Theorem~\ref{Theo-Lp-Error-Estimates} the following corollary directly
follows by estimating the series on the right hand side in the same way as 
in the proof of Corollary~\ref{Cor-I-In-Fehler-Estimate}.
\begin{cor} \label{Corollary-Lp-Error-Estimates}
	Let $p > 2$ and $\Se \in \mathbb{N}$. Then, for the approximation
	${I}_{(i,j)}^{(\Se)}(h)$  defined in \eqref{Approximation-vec-Version-hat-Iij} 
	of the iterated stochastic integral $I_{(i,j)}(h)$ for $1 \leq i, j \leq m$ it holds
	\begin{enumerate}[(i)]
		\item \label{Corollary-Lp-Error-Estimates-i}
		\begin{equation} \label{Cor-Lp-Fehler-Estimates-eqn1}
		\max_{1 \leq i,j \leq m} \big( \Erw \big( | I_{(i,j)}(h) - I_{(i,j)}^{(\Se)}(h) |^p 
		\big) \big)^{1/p}
		\leq c_{m,p} \frac{\sqrt{p-1} h}{\sqrt{3} \pi^{\frac{2p+1}{2p}} \Se} \, , 
		\end{equation}
		\item \label{Corollary-Lp-Error-Estimates-ii}
		\begin{equation} 
		\big( \Erw \big( \| I(h) - I^{(\Se)}(h) \|_F^p \big) \big)^{1/p}
		\leq c_{m,p} \frac{\sqrt{(p-1) (m^2-m)} h}{\sqrt{3} \pi^{\frac{2p+1}{2p}} 
		\Se} \, ,
		\end{equation}
	\end{enumerate}
	where the constant $c_{m,p}$ is the same as the one given in 
	Theorem~\ref{Theo-Lp-Error-Estimates}.
\end{cor}
\section{Simulation of iterated stochastic integrals}
\label{Sec:Simu-Alg}
For many applications like, e.g., stochastic models described by SDEs or SPDEs, 
one needs to simulate realizations of the approximate solutions. Here it has 
to be pointed out that, in general, the simulation problem is different from the
approximation problem. For the approximation problem considered in
Section~\ref{Sec:Alg} we make use of information about the realization 
of the driving Brownian motion like the values of the increments of the 
Brownian motion, of the Fourier coefficients and of
$\Psi^{1,(\Se)}$ in \eqref{RV-Psi-1-n-Exact} as well as 
of $\Psi^{2,(\Se)}$ in \eqref{RV-Psi-2-n-Exact}. This information has to be
provided or needs to be known and is a priori fixed for the
approximation problem. On the other hand, for the
simulation problem we are free to generate or choose a realization of the 
driving Brownian motion to be considered, i.e., to generate the whole necessary
information ourselves.
For the simulation problem, one only has to take care to sample from the 
correct distribution when this information is generated. 
\\ \\
Since the distribution of the iterated stochastic integrals is, in general, not 
explicitly known, it is necessary to apply an approximation algorithm.
Depending on the considered error criterion, the realizations generated by 
the approximation algorithm need to be close enough to corresponding 
realizations that may come into existence based on the exact distribution. 
This feature is guaranteed if the algorithm under consideration also
applies to the approximation problem, which is the case for algorithm
\eqref{Approximation-vec-Version-hat-Iij} proposed in Section~\ref{Sec:Alg}
as convergence in $L^p(\Omega)$-norm is proved. Therefore, we can now
describe how to apply this algorithm for simulating approximations for 
iterated stochastic integrals with some prescribed precision.
%
%
\subsection{Simulation algorithm}
\label{Sec:Simu-Alg-SubSec:SimAlg}
Assume that we want to simulate the twofold iterated stochastic integrals 
$I_{(i,j)}(h)$ together with the increments $\Delta W^i(h)$ for $1 \leq i,j \leq m$
and some $h>0$ such that $L^p$-accuracy of at least $\varepsilon>0$ is 
guaranteed, i.e., such that
\begin{align} \label{Sec:Simu-Error-bound}
	\max_{1 \leq i,j \leq m} \big( \Erw \big( | I_{(i,j)}(h) - I_{(i,j)}^{(\Se)}(h) |^p 
	\big) \big)^{1/p} \leq \varepsilon \, .
\end{align}
Therefore, we have to determine $\Se \in \mathbb{N}$ as small as possible 
under the condition that \eqref{Sec:Simu-Error-bound} is fulfilled. Let 
$\hat{c}_{m,p} = \frac{\sqrt{m}}{\sqrt{12} \, \pi}$ for $p=2$ due to
\eqref{Cor-I-In-Fehler-Estimate-eqn1} and
$\hat{c}_{m,p} = c_{m,p} \frac{\sqrt{p-1}}{\sqrt{3} \pi^{\frac{2p+1}{2p}}}$ if
$p>2$ due to \eqref{Cor-Lp-Fehler-Estimates-eqn1}. Then, it follows that
\begin{align*}
	\Se \geq \hat{c}_{m,p} \frac{h}{\varepsilon}
\end{align*}
has to be fulfilled and we choose $\Se = \lceil \hat{c}_{m,p} 
\frac{h}{\varepsilon} \rceil$. 
\\ \\
The algorithm for the simulation of $I_{(i,j)}(h)$
and $\Delta W^i(h)$ for $i,j \in \{1, \ldots, m\}$ is defined as follows:
%
%
\phantomsection \label{Sec:Algo1}\noindent
Let 
$h>0$ and $\Se \in \mathbb{N}$ be given.
\begin{sffamily}
	\begin{enumerate}
		\item  Simulate $V \sim \Nd(0_m,I_m)$ and let 
		\begin{equation*}
			\Delta W^i(h) = \sqrt{h} \, V_i
		\end{equation*}
		for $i = 1, \ldots, m$.
		\item For $k=1, \ldots, \Se$ simulate $X_{k}, Y_{k} \sim \Nd(0_m,I_m)$ and 
		calculate
		\begin{equation*}
			A_{(i,j)}^{(\Se)}(h) = 
			\frac{h}{2 \pi} \sum_{k=1}^{\Se} \frac{1}{k} \bigg( 
			X_{i,k} \bigg( Y_{j,k} - \frac{\sqrt{2}}{\sqrt{h}} \Delta W^j(h) \bigg) 
			- X_{j,k} \bigg( Y_{i,k} - \frac{\sqrt{2}}{\sqrt{h}} \Delta W^i(h) \bigg)
			\bigg)
		\end{equation*}
		for $1 \leq i < j \leq m$.
		\item Simulate $\Psi^{1,(\Se)} \sim \Nd(0_m,I_m)$ and compute
		\begin{equation*}
			A_{(i,j)}^{1,(\Se)}(h) = \frac{\sqrt{h}}{\sqrt{2} \pi} 
			\bigg( \frac{\pi^2}{6} - \sum_{k=1}^{\Se} \frac{1}{k^2} \bigg)^{1/2} 
			\big( \Delta W^i(h) \, \Psi_j^{1,(\Se)} - \Delta W^j(h) \, \Psi_i^{1,(\Se)} \big)
		\end{equation*}
		for $1 \leq i < j \leq m$.
		\item Let $M= \frac{m(m-1)}{2}$. Simulate $\Psi^{2,(\Se)} \sim \Nd(0_M,I_M)$ 
		and compute 
		\begin{equation*}
			A_{(i,j)}^{2,(\Se)}(h) = \frac{h}{\sqrt{2} \pi} 
			\bigg( \frac{\pi^2}{6} - \sum_{k=1}^{\Se} \frac{1}{k^2} \bigg)^{1/2} 
			\Psi_{r}^{2,(\Se)}
		\end{equation*}
		for $1 \leq i < j \leq m$ with $r = (i-1)m + j - \sum_{k=1}^i k$.
		\item Compute the approximation ${I}_{(i,j)}^{(\Se)}(h)$ of $I_{(i,j)}(h)$ 
		as
		\begin{equation*}
			{I}_{(i,j)}^{(\Se)}(h) = \frac{1}{2} \Delta W^i(h) \, \Delta W^j(h)
			+ A_{(i,j)}^{(\Se)}(h) +  A_{(i,j)}^{1,(\Se)}(h) + A_{(i,j)}^{2,(\Se)}(h)
		\end{equation*}
		and let 
		\begin{equation*}
			{I}_{(j,i)}^{(\Se)}(h) = \frac{1}{2} \Delta W^i(h) \, \Delta W^j(h)
		- A_{(i,j)}^{(\Se)}(h) -  A_{(i,j)}^{1,(\Se)}(h) - A_{(i,j)}^{2,(\Se)}(h)
		\end{equation*}		
		for $1 \leq i < j \leq m$ and further let ${I}_{(i,i)}^{(\Se)}(h) = \frac{1}{2} 
		\big( (\Delta W^i(h))^2 - h \big)$ for $i=1, \ldots, m$.
	\end{enumerate}
\end{sffamily}
%
%
%
Note that for the simulation of $I_{(i,j)}^{(\Se)}(h)$ and $\Delta W^i(h)$ for $i,j \in 
\{1, \ldots, m\}$ one has to simulate $2m (\Se+1) + \frac{m(m-1)}{2}$ 
realizations of independent and identically $\Nd(0,1)$ distributed random
variables. Moreover, it has to be pointed out that in contrast to the 
algorithm proposed by Wiktorsson~\cite{MR1843055} the presented algorithm
permits to simulate the iterated stochastic integrals $I_{(i,j)}^{(\Se)}(h)$ 
sequentially and even to increase the dimension $m$ during the simulation 
without recomputing the already simulated iterated stochastic integral. This
feature is due to the covariance matrix $\Sigma^{2, \infty}$ which is a 
diagonal matrix. As a result of this, the approximations for the truncation
terms $A_{(i,j)}^{2,(\Se)}(h)$ are stochastically independent for $i<j$.
\subsection{Computational cost}
\label{Sec:Simu-Alg-SubSec:CompCost}
Now, we take a closer look at the computational cost that we measure as
the number of independent realizations of standard Gaussian random variables
needed for one realization of $\Delta W(h)$ together with 
the matrix $I^{(\Se)}(h)$ for some $h>0$ in case of an $m$-dimensional 
Brownian motion. Given some prescribed error bound $\varepsilon$ for 
the $L^p(\Omega)$-error \eqref{Sec:Simu-Error-bound}, the 
improved algorithm proposed in 
Section~\ref{Sec:Simu-Alg-SubSec:SimAlg}, which is denoted as $\IA$, 
has computational cost
\begin{align*}
	\costs_{\IA}(\varepsilon) &= 2m \Big( \Big\lceil \hat{c}_{m,p} \frac{h}{\varepsilon} 
	\Big\rceil +1 \Big) + \frac{m (m-1)}{2}
\end{align*}
for the simulation of one realization of the increments $\Delta W^i(h)$ together
with the iterated stochastic integrals $I_{(i,j)}^{(\Se)}(h)$ for all $i,j =1, \ldots,m$. 
Especially, note that $\hat{c}_{m,2} = \frac{\sqrt{m}}{\sqrt{12} \, \pi}$ if $p=2$.
For a comparison, we consider the cost of the algorithm proposed by
Wiktorsson~\cite{MR1843055} denoted as $\WIK$, for which upper
$L^2(\Omega)$-error bounds are known. Therefore, choosing $\Se$ as in
\cite[(4.9)]{MR1843055} for this algorithm it holds that
\begin{align*}
	\costs_{\WIK}(\varepsilon) &= 2m \bigg( \Big\lceil \frac{\sqrt{5 (m-1)}
	 m}{\sqrt{24} \pi} \frac{h}{\varepsilon} \Big\rceil + \frac{1}{2} \bigg) 
 	+ \frac{m (m-1)}{2}
\end{align*}
if the error bound in \eqref{Sec:Simu-Error-bound} has to be fulfilled for $p=2$.
Finally, we consider algorithm \eqref{Iij-Truncated-Fourier-Series-Alg} denoted 
as $\KPW$ that is proposed in
\cite{MR1178485,MR1214374,Mil95}. For $p \geq 2$ it follows with
Proposition~\ref{Prop-Iij-FS-convergence-Lp} 
that, in order to fulfill the error bound
\eqref{Sec:Simu-Error-bound}, one has to choose $\Se \geq \big\lceil 
\frac{(p-1)^2}{2 \pi^2} \big( \Gamma \big( \frac{p}{2}+1 \big) \big)^{2/p}
\frac{h^2}{\varepsilon^2} \big\rceil$. Choosing $\Se = 
\big\lceil \frac{(p-1)^2}{2 \pi^2} \big( \Gamma \big( \frac{p}{2}+1 \big) \big)^{2/p}
\frac{h^2}{\varepsilon^2} \big\rceil$ results in the
computational cost
\begin{align*}
	\costs_{\KPW}(\varepsilon) &= 2m \bigg( 
	\Big\lceil \frac{(p-1)^2}{2 \pi^2} \Big( \Gamma \Big( \frac{p}{2}+1 \Big) \Big)^{2/p}
	\frac{h^2}{\varepsilon^2} \Big\rceil + 1 \bigg) \, .
\end{align*}
Comparing the computational costs of the algorithms $\IA$, $\WIK$ and $\KPW$ 
as $\varepsilon \to 0$, it follows that the costs of $\IA$ and $\WIK$ are of
order $\Oo( h \varepsilon^{-1} )$ while the cost of $\KPW$ is of order 
$\Oo(h^2 \varepsilon^{-2} )$ for some fixed $h>0$. 
Thus, algorithms $\IA$ and $\WIK$ possess a higher order of convergence
than algorithm $\KPW$ if their errors versus costs are compared.
Having a closer look at algorithms $\IA$ and $\WIK$,
there is asymptotically a reduction by the factor
$\sqrt{\frac{5}{2} m(m-1)}$ for the cost of algorithm $\IA$ compared to
the cost of $\WIK$ if the error estimates in \cite[Theorem 4.1]{MR1843055} are
applied in the case of $p=2$. It is worth mentioning that in case of the Frobenius 
norm as in
\eqref{Cor-I-In-Fehler-Estimate-eqn2} the computational cost for $\IA$ is
reduced by the factor $\sqrt{5}$ compared to algorithm $\WIK$. 
This reduction of the computational cost for algorithm $\IA$ originates in 
the exact approximation of the Fourier coefficients $a_{i,0}$ represented by 
the truncation term $\hat{R}^{1,(\Se)}$ that is conditionally Gaussian
distributed. Moreover,
it has to be pointed out that algorithm $\IA$ does not need the calculation
of the square root of a $M \times M$ covariance matrix as it is the case for
algorithm $\WIK$, see \cite[(4.7)]{MR1843055} because the covariance 
matrix $\Sigma^{2,\infty}$ in \eqref{Sigma-Infty} for algorithm $\IA$ 
is a multiple of the identity matrix
whereas the covariance matrix for algorithm $\WIK$ even
depends on $\Delta W(h)$ and thus needs to be recalculated for each 
realization. As a result of this, there is a substantial improvement in applying 
algorithm $\IA$ compared to the algorithm  $\WIK$ by 
Wiktorsson~\cite{MR1843055}, on the one hand by a reduction of the 
number of necessary 
realizations of Gaussian random variables and on the other hand by saving 
the calculation of the square root of the covariance matrix.
\\ \\
As an example, if one of the algorithms $\KPW$, $\WIK$ or $\IA$ is applied
together with a numerical scheme 
for the computation of $L^2$-approximations $Y^h$ of solutions $X$ for some 
SDE with 
a root mean square error (RMSE) of order $\Oo(h)$ w.r.t.\ step size $h$ like the 
Milstein scheme, i.e., such that 
\begin{align*}
	\big( \Erw ( \|X_T - Y^h_T \|^2 ) \big)^{1/2} = \Oo(h)
\end{align*}
at some time point $T>0$ as $h \to 0$,
then one has to choose $\Se$ such that the convergence rate
$\Oo(h)$ is preserved if the twofold iterated stochastic integrals in the numerical
scheme are replaced by the approximated ones. This means that the
RMSE for the approximated iterated stochastic integrals 
has to be of order $\Oo(h^{3/2})$, see \cite[Cor.~10.6.5]{MR1214374} or
\cite[Lem.~6.2]{Mil95}. Choosing $\varepsilon = h^{3/2}$ results in
$\Se = \big\lceil \frac{\sqrt{m}}{\sqrt{12} \pi} h^{-1/2} \big\rceil$ and the total
cost for the computation of one realization of $Y^h(T)$ that is based on approximations on $T/h$ time intervals amounts to
$\costs_{\IA}(h) = \Oo(h^{-3/2})$. Analogously, the computational cost for
algorithm $\WIK$ is $\costs_{\WIK}(h) = \Oo(h^{-3/2})$. In contrast to that,
the computational cost for algorithm $\KPW$ is $\costs_{\KPW}(h) =
\Oo(h^{-2})$,
which is of the same order as if one would apply the strong order $1/2$
Euler-Maruyama scheme with step size $h^2$ resulting in $\costs_{\EM}(h^2) 
= \Oo(h^{-2})$ in order to obtain the same RMSE of order $\Oo(h)$. Thus, 
instead of
applying, e.g., the Milstein scheme together with algorithm $\KPW$, one may
use the Euler-Maruyama scheme with step size $h^2$, which is easier
to implement and in general needs less computational effort. However, if the
Milstein scheme is combined with 
the algorithms $\IA$ or $\WIK$, then a higher order of convergence than that 
of the Euler-Maruyama scheme is attained, see also discussions in
\cite{MR2669396,MR1843055}. Especially, if the RMSE versus computational
costs are considered, then the resulting so-called effective order of convergence 
for the Milstein scheme together with algorithm $\IA$ or $\WIK$ is
$p_{\text{eff}}=3/2$,
whereas the effective order of convergence for the Milstein scheme together
with algorithm $\KPW$ is $p_{\text{eff}}=1/2$, which is the same as that for the 
Euler-Maruyama scheme. Thus, the algorithms $\IA$ and $\WIK$ allow to 
improve the order of convergence if they are combined with the Milstein scheme.
Since algorithm $\IA$ needs significantly less 
computational effort than algorithm $\WIK$ in order to guarantee the same
RMSE, the combination of the Milstein scheme with algorithm $\IA$
outperforms the combination of the Milstein scheme with algorithm $\WIK$.
\appendix
\section{Appendix: The covariance matrix}
\label{Appendix-A}
For the proofs of convergence, it is useful to have a closer look at the
matrix $H_m \Sigma(X_k) H_m^{\Tt}$. Observe, that we can factorize this matrix 
by
\begin{align*}
H_m \Sigma(X_k) H_m^{\Tt} &= H_m (P_m - I_{m^2}) (I_m \otimes (X_k X_k^{\Tt}))
(P_m - I_{m^2})^{\Tt} H_m^{\Tt} \\
&= H_m (P_m-I_{m^2}) (I_m \otimes X_k) \big[ H_m (P_m-I_{m^2}) (I_m \otimes 
X_k) \big]^{\Tt} \\
&= H_m ( X_k \otimes I_m - I_m \otimes X_k) \big[ H_m ( X_k \otimes I_m 
- I_m \otimes X_k) \big]^{\Tt} \, .
\end{align*}
Now, we can easily calculate the $m^2 \times m$ matrix
\begin{align*}
	X_k \otimes I_m - I_m \otimes X_k = \begin{pmatrix} 
	0 & & & & & \\
	-X_{2,k} & X_{1,k} & & & & \\
	-X_{3,k} & & X_{1,k} & & & \\
	-X_{4,k} & & & X_{1,k} & & \\
	\vdots & & & & \ddots & \\
	-X_{m,k} & & & & & X_{1,k} \\
	\cdashline{1-6}
	X_{2,k} & -X_{1,k} & & & & \\
	& 0 & & & & \\
	& -X_{3,k} & X_{2,k} & & & \\
	& -X_{4,k} & & X_{2,k} & & \\
	& \vdots & & & \ddots & \\
	& -X_{m,k} & & & & X_{2,k} \\
	\cdashline{1-6}
	\vdots & & & & & \\
	\cdashline{1-6}
	X_{l,k} & & & -X_{1,k} & & \\
	& X_{l,k} & & -X_{2,k} & & \\
	& & \ddots & \vdots & & \\
	& & & 0 & & \\
	& & & \vdots & \ddots & \\
	& & & -X_{m,k} & & X_{l,k} \\
	\cdashline{1-6}
	\vdots & & & & & \\
	\cdashline{1-6}
	X_{m,k} & & & & & -X_{1,k} \\
	& X_{m,k} & & & & -X_{2,k} \\
	& & X_{m,k} & & & -X_{3,k} \\
	& & & \ddots & & \vdots \\
	& & & & X_{m,k} & -X_{m-1,k} \\
	& & & & & 0 
	\end{pmatrix} \, .
\end{align*}
The selection matrix $H_m$ applied to $X_k \otimes I_m - I_m \otimes X_k$ 
picks out the $M$ rows that we obtain by deleting in the first block the first 
row, in the second block the first two rows, in the $l$-th block the first $l$
rows and so on, until the $m-1$-th block, were we only take the last row while
the $m$-th block is deleted completely. Finally, we have to multiply 
$H_m (X_k \otimes I_m - I_m \otimes X_k)$ with itself transposed which results
in the symmetric $M \times M$ block matrix 
\begin{align} \label{Matrix-Rep-HSXH}
	H_m \Sigma(X_k) H_m^{\Tt} &= \begin{pmatrix}
		B_{1,1}^k & B_{1,2}^k & \cdots & B_{1,m-1}^k \\
		B_{2,1}^k & B_{2,2}^k & \cdots & B_{2,m-1}^k \\
		\vdots & \vdots & \ddots & \vdots \\
		B_{m-1,1}^k & B_{m-1,2}^k & \cdots & B_{m-1,m-1}^k
	\end{pmatrix} \, .
\end{align}
For $l \in \{1, \ldots, m-1\}$ the $l$-th diagonal block can be calculated as the
symmetric $(m-l) \times (m-l)$ matrix 
\begin{align*}
	B_{l,l}^k = \begin{pmatrix}
			X_{l,k}^2 + X_{l+1,k}^2 & X_{l+1,k} X_{l+2,k} & X_{l+1,k} X_{l+3,k} & 
			\cdots \cdots & \cdots \cdots & X_{l+1,k} X_{m,k} \\
			X_{l+2,k} X_{l+1,k} & X_{l,k}^2 + X_{l+2,k}^2 & X_{l+2,k} X_{l+3,k} &
			\cdots \cdots & \cdots \cdots & X_{l+2,k} X_{m,k} \\
			X_{l+3,k} X_{l+1,k} & X_{l+3,k} X_{l+2,k} & X_{l,k}^2 + X_{l+3,k}^2 &
			\cdots \cdots & \cdots \cdots & X_{l+3,k} X_{m,k} \\
			\vdots & \vdots & & \ddots & & \vdots \\
			\vdots & \vdots & & & \ddots & \vdots \\
			X_{m,k} X_{l+1,k} & X_{m,k} X_{l+2,k} & X_{m,k} X_{l+3,k} & \cdots \cdots &
			\cdots \cdots & X_{l,k}^2 + X_{m,k}^2
		\end{pmatrix} \, .
\end{align*}
For $1 \leq r < s \leq m-1$ the non-diagonal block at position
$(r,s)$ is the $(m-r) \times (m-s)$ matrix that can be calculated as
\begin{align*}
	B_{r,s}^k = \begin{pmatrix}
		0_{(s-r-1) \times (m-s)} \\
		-b_{r,s}^k \\
		d_{r,s}^k
	\end{pmatrix}
\end{align*}
with the $1 \times (m-s)$ vector $b_{r,s}^k = (X_{r,k} X_{s+1,k}, X_{r,k} X_{s+2,k}, 
X_{r,k} X_{s+3,k}, \ldots, X_{r,k} X_{m,k} )$ and the $(m-s) \times (m-s)$ 
diagonal matrix $d_{r,s}^k = \diag ( X_{r,k} X_{s,k}, \ldots, X_{r,k} X_{s,k} )$. 
Further, it holds $B_{s,r}^k = (B_{r,s}^k)^{\Tt}$.
With \eqref{Matrix-Rep-HSXH} it follows that the matrix $H_m \Sigma(X_k)
H_m^{\Tt}$
is a symmetric $M \times M$ matrix such that $(H_m \Sigma(X_k) H_m^{\Tt})_{p,q}
\in \{ X_{i,k}^2 + X_{j,k}^2, \pm X_{i,k} X_{j,k}, 0 \}$ for some $i \neq j$. 
Moreover, each row (column) has exactly $2m-4$ matrix entries of type
$\pm X_{i,k} X_{j,k}$ and one diagonal entry of type $X_{i,k}^2 + X_{j,k}^2$
for some $i,j \in \{1, \ldots, m\}$ with $i \neq j$, respectively,
while the remaining $M-2m+3$ entries in each row (column) are equal to $0$.
\\ \\
%
%
\textbf{Acknowledgement}
Funding and support by
the Graduate School for Computing in Medicine and Life Sciences funded by
Germany's Excellence Initiative [DFG GSC 235/2] is gratefully acknowledged. 
%
%
\phantomsection
\addcontentsline{toc}{section}{References}
\bibliographystyle{plainurl}
\bibliography{BibPaper}

\providecommand{\noopsort}[1]{}
\begin{thebibliography}{10}

\bibitem{MR3618289}
R.~Bhattacharya and E.~C. Waymire.
\newblock {\em A basic course in probability theory}.
\newblock Universitext. Springer, Cham, second edition, 2016.
\newblock \href {https://doi.org/10.1007/978-3-319-47974-3}
  {\path{doi:10.1007/978-3-319-47974-3}}.

\bibitem{MR2767184}
E.~\c{C}\i nlar.
\newblock {\em Probability and stochastics}, volume 261 of {\em Graduate Texts
  in Mathematics}.
\newblock Springer, New York, 2011.
\newblock \href {https://doi.org/10.1007/978-0-387-87859-1}
  {\path{doi:10.1007/978-0-387-87859-1}}.

\bibitem{MR1284705}
J.~G. Gaines and T.~J. Lyons.
\newblock Random generation of stochastic area integrals.
\newblock {\em SIAM J. Appl. Math.}, 54(4):1132--1146, 1994.
\newblock \href {https://doi.org/10.1137/S0036139992235706}
  {\path{doi:10.1137/S0036139992235706}}.

\bibitem{HaRoe20arXiv2006.08275v1}
C.~{\noopsort{Hallern}}{von Hallern} and A.~{R{\"o}{\ss}ler}.
\newblock {A Derivative-Free Milstein Type Approximation Method for SPDEs
  covering the Non-Commutative Noise case}.
\newblock {\em ArXiv e-prints, v1}, June 2020.
\newblock \href {http://arxiv.org/abs/2006.08275v1}
  {\path{arXiv:2006.08275v1}}.

\bibitem{MCQMC2018}
C.~{\noopsort{Hallern}}{von Hallern} and A.~{R{\"o}{\ss}ler}.
\newblock {An Analysis of the Milstein Scheme for SPDEs without a Commutative
  Noise Condition}.
\newblock In B.~Tuffin and P.~L'Ecuyer, editors, {\em {Monte Carlo and
  Quasi-Monte Carlo Methods, MCQMC 2018}}, volume 324 of {\em Springer
  Proceedings in Mathematics \& Statistics}, pages 503--521. Springer, Cham,
  2020.
\newblock \href {https://doi.org/10.1007/978-3-030-43465-6_25}
  {\path{doi:10.1007/978-3-030-43465-6_25}}.

\bibitem{MR1956867}
J.~Jacod and P.~Protter.
\newblock {\em Probability essentials}.
\newblock Universitext. Springer-Verlag, Berlin, second edition, 2003.
\newblock \href {https://doi.org/10.1007/978-3-642-55682-1}
  {\path{doi:10.1007/978-3-642-55682-1}}.

\bibitem{MR1474726}
S.~Janson.
\newblock {\em Gaussian {H}ilbert spaces}, volume 129 of {\em Cambridge Tracts
  in Mathematics}.
\newblock Cambridge University Press, Cambridge, 1997.
\newblock \href {https://doi.org/10.1017/CBO9780511526169}
  {\path{doi:10.1017/CBO9780511526169}}.

\bibitem{MR3320928}
A.~Jentzen and M.~R{\"o}ckner.
\newblock A {M}ilstein scheme for {SPDE}s.
\newblock {\em Found. Comput. Math.}, 15(2):313--362, 2015.
\newblock \href {https://doi.org/10.1007/s10208-015-9247-y}
  {\path{doi:10.1007/s10208-015-9247-y}}.

\bibitem{MR1214374}
P.~E. Kloeden and E.~Platen.
\newblock {\em Numerical solution of stochastic differential equations},
  volume~23 of {\em Applications of Mathematics (New York)}.
\newblock Springer-Verlag, Berlin, 1992.
\newblock \href {https://doi.org/10.1007/978-3-662-12616-5}
  {\path{doi:10.1007/978-3-662-12616-5}}.

\bibitem{MR1178485}
P.~E. Kloeden, E.~Platen, and I.~W. Wright.
\newblock The approximation of multiple stochastic integrals.
\newblock {\em Stochastic Anal. Appl.}, 10(4):431--441, 1992.
\newblock \href {https://doi.org/10.1080/07362999208809281}
  {\path{doi:10.1080/07362999208809281}}.

\bibitem{MR3949104}
C.~Leonhard and A.~R\"{o}{\ss}ler.
\newblock Iterated stochastic integrals in infinite dimensions: approximation
  and error estimates.
\newblock {\em Stoch. Partial Differ. Equ. Anal. Comput.}, 7(2):209--239, 2019.
\newblock \href {https://doi.org/10.1007/s40072-018-0126-9}
  {\path{doi:10.1007/s40072-018-0126-9}}.

\bibitem{MR520247}
J.~R. Magnus and H.~Neudecker.
\newblock The commutation matrix: some properties and applications.
\newblock {\em Ann. Statist.}, 7(2):381--394, 1979.
\newblock URL: \url{https://www.jstor.org/stable/2958818}.

\bibitem{MR3178331}
S.~J.~A. Malham and A.~Wiese.
\newblock Efficient almost-exact {L}\'{e}vy area sampling.
\newblock {\em Statist. Probab. Lett.}, 88:50--55, 2014.
\newblock \href {https://doi.org/10.1016/j.spl.2014.01.022}
  {\path{doi:10.1016/j.spl.2014.01.022}}.

\bibitem{Mil75}
G.~N. Milstein.
\newblock {Approximate Integration of Stochastic Differential Equations}.
\newblock {\em Theory Probab. Appl.}, 19:557--562, 1975.
\newblock \href {https://doi.org/10.1137/1119062} {\path{doi:10.1137/1119062}}.

\bibitem{Mil95}
G.~N. Milstein.
\newblock {\em {Numerical Integration of Stochastic Differential Equations}}.
\newblock Kluwer Academic Publishers, Dordrecht, 1995.

\bibitem{MR2200233}
D.~Nualart.
\newblock {\em The {M}alliavin calculus and related topics}.
\newblock Probability and its Applications (New York). Springer-Verlag, Berlin,
  second edition, 2006.

\bibitem{MR1331198}
I.~Pinelis.
\newblock Optimum bounds for the distributions of martingales in {B}anach
  spaces.
\newblock {\em Ann. Probab.}, 22(4):1679--1706, 1994.
\newblock URL: \url{http://www.jstor.org/stable/2244912}.

\bibitem{PleisDiss20}
J.~Pleis.
\newblock {$L^p$ and Pathwise Convergence of the Milstein Scheme for Stochastic
  Delay Differential Equations}.
\newblock {\em Ph.D. thesis, Institute of Mathematics, Universit\"at zu
  L\"ubeck}, 2020.

\bibitem{PlRoe20arXiv2006.16048v1}
J.~Pleis and A.~R{\"o}{\ss}ler.
\newblock {A Note on Some Martingale Inequalities}.
\newblock {\em ArXiv e-prints, v1}, June 2020.
\newblock \href {http://arxiv.org/abs/2006.16048} {\path{arXiv:2006.16048}}.

\bibitem{MR2669396}
A.~R\"{o}{\ss}ler.
\newblock Runge-{K}utta methods for the strong approximation of solutions of
  stochastic differential equations.
\newblock {\em SIAM J. Numer. Anal.}, 48(3):922--952, 2010.
\newblock \href {https://doi.org/10.1137/09076636X}
  {\path{doi:10.1137/09076636X}}.

\bibitem{MR2116088}
D.~M. Stump and J.~M. Hill.
\newblock On an infinite integral arising in the numerical integration of
  stochastic differential equations.
\newblock {\em Proc. R. Soc. Lond. Ser. A Math. Phys. Eng. Sci.},
  461(2054):397--413, 2005.
\newblock \href {https://doi.org/10.1098/rspa.2004.1379}
  {\path{doi:10.1098/rspa.2004.1379}}.

\bibitem{MR1843055}
M.~Wiktorsson.
\newblock Joint characteristic function and simultaneous simulation of iterated
  {I}t\^{o} integrals for multiple independent {B}rownian motions.
\newblock {\em Ann. Appl. Probab.}, 11(2):470--487, 2001.
\newblock \href {https://doi.org/10.1214/aoap/1015345301}
  {\path{doi:10.1214/aoap/1015345301}}.

\end{thebibliography}
\end{document}